\documentclass[11pt]{amsart}

\usepackage[foot]{amsaddr}
\makeatletter
\renewcommand{\email}[2][]{%
  \ifx\emails\@empty\relax\else{\g@addto@macro\emails{,\space}}\fi%
  \@ifnotempty{#1}{\g@addto@macro\emails{\textrm{(#1)}\space}}%
  \g@addto@macro\emails{#2}%
}
\makeatother

\makeatletter
\@namedef{subjclassname@2020}{\textup{2020} Mathematics Subject Classification}
\makeatother

\usepackage{amsthm,amsmath,amssymb,enumerate,hyperref,fontenc,graphicx}
\usepackage{bbm}
\usepackage{multirow}
\AtBeginDocument{
  \renewcommand{\setminus}{\mathbin{\backslash}}%
}
\usepackage{stmaryrd}
\usepackage{tikz,tikz-cd}
\usetikzlibrary{matrix,arrows,decorations.markings,calc}

\usepackage{pgfplots}
\pgfplotsset{
        compat=1.16,
    }
    
\theoremstyle{plain}
\newtheorem{theorem}{Theorem}[section]
\newtheorem{lemma}[theorem]{Lemma}
\newtheorem{proposition}[theorem]{Proposition}

\theoremstyle{definition}
\newtheorem{definition}[theorem]{Definition}
\newtheorem{example}[theorem]{Example}
\newtheorem{remark}[theorem]{Remark}
\newtheorem{algorithm}[theorem]{Algorithm}

\newcommand{\mbbZ}{{\mathbb Z}}

\newcommand{\mcalC}{{\mathcal C}}

\newcommand{\mcalX}{{\mathcal X}}

\newcommand{\mfrako}{{\mathfrak o}}

\newcommand{\va}{{\mathbf a}}

\newcommand{\ve}{{\mathbf e}}

\newcommand{\vE}{{\mathbf E}}
\newcommand{\vF}{{\mathbf F}}

\DeclareMathOperator{\spec}{\operatorname{Spec}}

\DeclareMathOperator{\ab}{\operatorname{ab}}

\DeclareMathOperator{\ec}{\operatorname{ec}}
\DeclareMathOperator{\trace}{\operatorname{Tr}}
\DeclareMathOperator{\Aut}{\operatorname{Aut}}

\DeclareMathOperator{\outdeg}{\operatorname{outdeg}}
\DeclareMathOperator{\In}{\operatorname{In}}
\DeclareMathOperator{\red}{\operatorname{red}}
\DeclareMathOperator{\supp}{\operatorname{supp}}

\numberwithin{equation}{section}

\begin{document}

\title{On a Trace Formula for Counting Eulerian Cycles}
\author{Ye Luo} 
\address[A1]{School of Informatics, Xiamen University, Xiamen, Fujian 361005, China}
\email[A1]{luoye@xmu.edu.cn}

\subjclass[2020]{05C50, 05C38}
\keywords{Eulerian cycles, trace formulas, homological spectral graph theory}

\date{}

\begin{abstract}
Counting Eulerian cycles in undirected graphs is a classical problem in enumerative combinatorics and is \#P-complete in general.  We connect this problem with homological spectral graph theory and prove an explicit trace formula for the number of Eulerian cycles.  For an Eulerian graph $G$ with $m$ edges and genus $g$, the number $\ec(G)$ is expressed as a signed average of traces of the $m$-th powers of twisted vertex adjacency matrices $A_\gamma$ and twisted edge adjacency matrices $B_\gamma$, indexed by $H_1(G,\mbbZ/2\mbbZ)$.  The formula yields a fixed-parameter tractable exact algorithm with parameter $g$: after the $2^g$ twists are fixed, the remaining work is polynomial in the graph size.  Since \#P-completeness rules out a polynomial-time exact algorithm for arbitrary graphs unless $\mathrm{FP}=\#\mathrm{P}$, this confines the unavoidable complexity to the homological parameter in a precise sense.  Finally, we develop symmetry reductions coming from spectral antisymmetry and graph automorphisms, often decreasing the number of distinct twisted spectra that must be evaluated.
\end{abstract}

\maketitle
\section{Introduction}
An Eulerian circuit, or tour, is a closed traversal of a graph $G$ that uses every edge exactly once.  This concept originates in Euler's solution of the Seven Bridges of K\"onigsberg problem and is one of the starting points of graph theory.  A connected graph is Eulerian if and only if every vertex has even degree.  This existence criterion is accompanied by Hierholzer's linear-time algorithm, which finds an Eulerian circuit in time $O(|E|)$ when one exists.

The enumeration problem is much harder.  We count Eulerian cycles, i.e., Eulerian circuits up to cyclic translation.  For directed graphs, the counting problem is governed by the BEST theorem~\cite{AB1951circuits} (Theorem~\ref{T:best}), named after de Bruijn, van Aardenne-Ehrenfest, Smith, and Tutte.  It reduces the number of directed Eulerian cycles to rooted arborescence counts, which are computable in polynomial time by the directed matrix-tree theorem.

For undirected graphs, however, no such efficient method is known. By reduction to the problem of counting Eulerian orientations~\cite{MW1996on}, Brightwell and Winkler~\cite{BW2005counting} showed that the problem of counting Eulerian cycles for general graphs is \#P-complete. Later on, the same counting problem was proved to be \#P-complete for planar graphs by Creed~\cite{C2010counting} and for 4-regular graphs by Ge and {\v{S}}tefankovi{\v{c}}~\cite{GS2012the}.  The class \#P, introduced by Valiant~\cite{V1979complexity}, is the counting analogue of NP; in particular, a polynomial-time algorithm for a \#P-complete counting problem would imply polynomial-time algorithms for the corresponding NP-complete decision problems.  Consequently, unless $\mathrm{FP}=\#\mathrm{P}$, one cannot expect a general exact formula whose direct evaluation is polynomial-time on all graphs.  A realistic goal is therefore to identify natural parameters that contain the hardness and to make the remaining computation polynomial.

Given this \#P-completeness result, exact counting has mainly progressed in restricted regimes: asymptotic formulas for dense graph families, exact formulas or algorithms for structured graph classes, and general upper and lower bounds.  Asymptotic formulas have been found for complete graphs by McKay and Robinson~\cite{MR1998asymptotic}, for complete bipartite graphs~\cite{I2009asymptotic}, and for graphs with large algebraic connectivity by Isaev~\cite{I2011asymptotic}.  Chebolu, Cryan, and Martin gave a polynomial-time algorithm for generalized series-parallel graphs~\cite{CCM2012Exact}.  More recently, Punzi, Conte, Grossi, and Rizzi established upper and lower bounds for general graphs~\cite{P2023bounding, PCGR2024refined}.

Exact enumeration is also useful as a quantitative tool in structured applications.  In genome assembly, de Bruijn-type graphs model sequence reconstruction from local read data~\cite{PTW2001eulerian,CPT2011debruijn}; the number of Eulerian cycles can therefore measure reconstruction ambiguity.  In routing and inspection problems, Euler tours underlie the classical Chinese postman framework~\cite{EJ1973matching}, so exact counts measure the redundancy of complete traversals using every edge once.  Exact values on tractable families also provide benchmarks for approximation, sampling, and heuristic counting methods~\cite{C2010counting}.  These applications do not contradict the \#P-completeness barrier; rather, they motivate exact formulas that are effective on graph classes with additional structure or small topological complexity.

This paper develops a homological spectral approach to exact Eulerian enumeration.  The framework of Luo and Roy~\cite{LR2024spectral} counts circuits in specified homology classes by applying character sums to twisted adjacency matrices.  The key observation here is that an Eulerian circuit in an undirected graph is exactly a closed walk of length $m$ whose mod-$2$ abelianization is the all-one circulation.  This converts Eulerian enumeration into a homology-class counting problem and leads to an explicit trace formula.

Throughout this paper, a \emph{graph} means a connected finite graph allowing loops and multiple edges, and  a \emph{simple graph} means a connected finite graph without any loops or multiple edges. 

Consider a graph $G$ with vertex set $V(G)=\{v_1,\cdots,v_n\}$ and edge set $E(G)=\{e_1,\cdots,e_m\}$. The \emph{genus} $g$ of $G$ is defined as the first Betti number of $G$, i.e., $g=m-n+1$ where $m=|E|$ and $n=|V|$.  Assigning a direction to each edge gives an \emph{orientation} $\mfrako$ of $G$. For each edge $e_i\in E$, we associate two oriented edges $\ve_i$ and $\ve_i^{-1}$ which are positively and negatively oriented with respect to $\mfrako$. Let $\vE_\mfrako(G) = \{\ve_1,\cdots,\ve_m\}$ be the set of all positively oriented edges of $G$ and  $\vE(G)= \{\ve_1,\cdots,\ve_m,\ve_1^{-1},\cdots,\ve_m^{-1}\}=\{\ve_1,\cdots,\ve_m,\ve_{m+1},\cdots,\ve_{2m}\}$ be the set of all oriented edges of $G$. In particular, $G_\mfrako = (V(G),\vE_\mfrako(G))$ is the \emph{digraph with respect to the orientation $\mfrako$}, and we say $G$ is the \emph{underlying graph} of the digraph $G_\mfrako$. For a subgraph $H$ of $G$, denote by $E(H)$ the set of edges of $H$ and $\vE_\mfrako(H)$ the set of edges of $H$ positively oriented with respect to $\mfrako$. Also, we write $E^c(H):=E(G)\setminus E(H)$ and $\vE^c_\mfrako(H):=\vE_\mfrako(G)\setminus \vE_\mfrako(H)$. 

For each oriented edge $\ve\in\vE(G)$, let $\ve(0)$ and $\ve(1)$ be the \emph{initial} and \emph{terminal} vertices of $\ve$ respectively. Then the \emph{inverse} of $\ve$, denoted by $\ve^{-1}$, satisfies $\ve^{-1}(0)=\ve(1)$ and $\ve^{-1}(1)=\ve(0)$. Note that if $\ve$ is a loop, then $\ve(0)=\ve(1)$. For vertices $v,w\in V(G)$, a \emph{walk} $P$ from $v$ to $w$ of length $l$ is a sequence of (not necessarily distinct) oriented edges  $\va_1\cdots\va_l$ such that $\va_1(0)=v$, $\va_l(1)=w$, and $\va_{i+1}(0)=\va_{i}(1)$ for $i=1,\cdots,l-1$.  If  $\va_l=\va_1^{-1}$, then we say $P$ has a \emph{tail}, and if $\va_{i+1}=\va_i^{-1}$ for some $1\leq i\leq l-1$, then we say $P$ has a \emph{backtrack}. A \emph{closed walk} at the \emph{base vertex} $v$ is a walk from $v$ to $v$. We also call a closed walk with no tail and no backtracks a \emph{circuit}. Denote by  $\mcalC(G)$ (respectively $\widetilde{\mcalC}(G)$) the set of all circuits (respectively all closed walks) on $G$, and by $\mcalC_l(G)$ (respectively $\widetilde{\mcalC}_l(G)$) the set of all circuits (respectively all closed walks) of length $l$ on $G$. We call the equivalence classes of circuits under translation \emph{cycles}\footnote{Eulerian cycles are also referred to as Eulerian circuits or tours in the literature. In this paper, we distinguish between the terms of circuits and cycles, with circuits having a specified initial vertex while cycles do not.}.

For a graph $G=(V,E)$ of genus $g$, let $T$ be a spanning tree of $G$. Without loss of generality, we may assume $E^c(T) = \{e_1,\cdots,e_g\}$. Let $M_2(E^c(T) )$ be the free $\mbbZ/2\mbbZ$-module on $E^c(T)$. For each $\gamma = c_1\cdot e_1+\cdots c_g\cdot e_g\in M_2(E^c(T) )$ with $c_1,\cdots, c_g\in \mbbZ/2\mbbZ$, let $\sigma(\gamma):=\sum_{i=1}^g c_i$. 

The main result of this paper is the following explicit formula of counting the number of Eulerian cycles on $G$, denoted by $\ec(G)$, using homological spectral graph theory. 

\begin{theorem} \label{T:main}
Let $G$ be an Eulerian graph of genus $g$ with $m$ edges, and $T$ a spanning tree of $G$. Then we have the following trace formula for the number of Eulerian cycles on $G$:
\begin{align*}
\ec(G) &= \frac{1}{m\cdot 2^g}\sum_{\gamma\in M_2(E^c(T))}(-1)^{\sigma(\gamma)}\trace\left(B_\gamma^m\right) \\
&= \frac{1}{m\cdot 2^g}\sum_{\gamma\in M_2(E^c(T))}(-1)^{\sigma(\gamma)}\trace\left(A_\gamma^m\right).
\end{align*}
\end{theorem}

Here,  $\trace(M)$ denotes the trace of the square matrix  $M$, and $A_\gamma$ and $B_\gamma$ are respectively the twisted vertex and edge adjacency matrices with respect to $\gamma$, as precisely defined in Definition~\ref{D:twist}. 

We also provide a similar formula for directed Eulerian graphs in Theorem~\ref{T:directed_eulerian}. In addition, we show that computation can be reduced by using symmetries of twisted adjacency matrices: spectral antisymmetry (Theorem~\ref{T:red_anti}), graph automorphisms (Theorem~\ref{T:red_counting_eulerian}), and their combined effects (Theorem~\ref{T:red_counting_eulerian_combined}).  We include explicit algorithmic consequences: the trace formula is fixed-parameter tractable in the genus, and once the twisted traces are known the full distribution of circuit counts over $H_1(G,\mbbZ/2\mbbZ)$ is obtained by a Fast Walsh--Hadamard transform.  In view of the \#P-completeness barrier, this should be understood as a near-best possible general exact result in the following limited but meaningful sense: the formula does not remove exponential dependence altogether, but it isolates that dependence in the topological quantity $g$ and leaves only polynomial work for fixed genus.

Algorithmically, the trace formulas should be read as character-sum algorithms over the $2^g$ mod-$2$ homological twists.  In Theorems~\ref{T:counting0} and \ref{T:counting}, the symbol $l$ denotes the prescribed length of the closed walks or circuits being counted in a homology class.  For Eulerian-cycle counting in Theorem~\ref{T:main}, this length is not an additional parameter: it is specialized to $l=m$, because an Eulerian circuit uses all $m$ edges exactly once.  Thus the direct method has one trace computation for each of the $2^g$ twists, followed by polynomial-time matrix operations for each twist.  Section~\ref{S:algorithmic} spells out this interpretation as a baseline algorithm and records the corresponding implementation costs.

\begin{remark}
The factor $2^g$ is the size of the mod-$2$ homology character group in the direct character-sum evaluation.  The remaining polynomial factor can be improved in implementations.  For a square matrix $M$, the traces $\trace(M^k)$ can be recovered from the characteristic polynomial of $M$ by Newton identities, so one may compute $\trace(B_\gamma^m)$ or $\trace(A_\gamma^m)$ without explicitly forming the full $m$-th power.  Since $A_\gamma$ is an $n$ by $n$ Hermitian matrix while $B_\gamma$ is a $2m$ by $2m$ matrix that is generally nonnormal, the vertex formulation is often preferable for exact numerical implementation.  In settings where a twisted Bass--Ihara determinant formula is available, the edge-matrix trace can also be accessed through a determinant calculation involving the twisted vertex matrix and the degree matrix.  We keep the trace form in the main statements because it is direct, uniform, and best suited for the homological Fourier argument.
\end{remark}

\section{Preliminaries}
In this section, we consider a graph $G$ of genus $g$, a spanning tree $T$ of $G$, and an orientation $\mfrako$ of $G$. Again, let $V(G)=\{v_1,\cdots,v_n\}$, $E(G)=\{e_1,\cdots,e_m\}$, $\vE_\mfrako(G) = \{\ve_1,\cdots,\ve_m\}$ and $\vE(G)= \{\ve_1,\cdots,\ve_m,\ve_1^{-1},\cdots,\ve_m^{-1}\}$ be the vertex set, the edge set, the set of all positively oriented edges with respect to $\mfrako$, and the set of all oriented edges of $G$. Let $\epsilon_t=\exp(2\pi i/t)$ be a primitive $t$-th root of unity.

\subsection{Homology groups of graphs}
For a commutative ring $R$ with multiplicative identity $1_R$, let $C_1(G,R)$ be the free $R$-module on  $\vE(G)$ modulo $1\cdot \ve^{-1} = (-1)\cdot \ve$ for all $\ve\in \vE$. 
Note that $C_1(G,R)$ can be identified with the free $R$-module on $\vE_\mfrako$ for any orientation $\mfrako$. For $\alpha = \sum_{\ve\in \vE_\mfrako(G)}a_\ve\cdot \ve$ and $\beta = \sum_{\ve\in \vE_\mfrako(G)}b_\ve\cdot \ve$ in $C_1(G,R)$, write  $\langle\alpha,\beta\rangle := \sum_{\ve\in \vE_\mfrako(G)} a_\ve b_\ve$. 

\begin{definition}
An \emph{$R$-circulation} (or an $R$-homology class) on $G$ is an element $\alpha = a_1\cdot \ve_1+\cdots+a_m\cdot \ve_m$ of $C_1(G,R)$ satisfying the balancing condition: for each vertex $v\in V$, $\sum_{\ve_i(0)=v} a_i = \sum_{\ve_i(1)=v} a_i$. The \emph{first homology group} of $G$ over $R$, denoted by $H_1(G,R)$, is a submodule of $C_1(G,R)$ consisting of all $R$-circulations on $G$. 
For each walk $P=\va_1\cdots \va_l$ on $G$, we say $P^{\ab}=\va_1+\cdots+\va_l \in C_1(G,\mbbZ)$ is the \emph{abelianization} of $P$. 
\end{definition}

An obvious fact is that the abelianization $C^{\ab}$ of a closed walk $C$ is always a $\mbbZ$-circulation.

Now let us consider another type of submodules of $C_1(G,R)$. For a finite set $S$, denote by $M_R(S)$ the free $R$-module on $S$. In this sense, we have $C_1(G,R)=M_R(\vE_\mfrako(G))\simeq R^m$. For a subset $\vF$ of $\vE_\mfrako(G)$, we can also identify $M_R(\vF)$ as a submodule of $C_1(G,R)$, i.e., $\sum_{\ve\in \vF}c_\ve\cdot \ve = \sum_{\ve\in \vF}c_\ve\cdot \ve+\sum_{\ve\in \vE_\mfrako(G)\setminus \vF} 0\cdot \ve$ for all $c_\ve\in R$. On the other hand, the restriction of $\alpha = \sum_{\ve\in \vE_\mfrako(G) }c_\ve\cdot \ve \in C_1(G,R)$ to $\vF$ is naturally defined as $\alpha|_{\vF} := \sum_{\ve\in \vF}c_\ve\cdot \ve\in M_R(\vF)$.

For a spanning tree $T$ of $G$, it is clear that $M_R(\vE_\mfrako^c(T))\simeq R^g$ as $g=m-n+1$. The following lemma implies that $H_1(G,R)$ is also isomorphic to $R^g$. 

\begin{lemma} \label{L:rhoT}
The restriction homomorphism $\rho_T:H_1(G,R) \xrightarrow{\sim} M_R(\vE_\mfrako^c(T)) $ sending $\alpha\in H_1(G,R)$ to $\alpha|_{\vE_\mfrako^c(T)}\in M_R(\vE_\mfrako^c(T))$  is an isomorphism. 
\end{lemma}

\begin{proof}
Suppose $\vE_\mfrako^c(T)=\{\ve_1,\cdots,\ve_g\}$. 
For each $\ve_i\in \vE_\mfrako^c(T)$, there is a unique circuit $C_i$ which is the concatenation of $\ve_i$ with the unique nonbacktracking path on $T$ from $\ve_i(1)$ to $\ve_i(0)$. 
Then their abelianizations  $C^{\ab}_1,\cdots,C^{\ab}_g$ must be $\mbbZ$-circulations in $H_1(G,\mbbZ)$. Since $H_1(G,R)=R\otimes H_1(G,\mbbZ)$, we see that $1_R\otimes C^{\ab}_i\in H_1(G,R)$ and $\rho_T(1_R\otimes C^{\ab}_i)=1_R\cdot \ve_i$ where $1_R$ is the multiplicative identity of $R$. This means that $\rho_T$ is surjective. 

On the other hand, the injectivity of $\rho_T$ follows from the fact that the kernel of $\rho_T$ is trivial, i.e., $\sum_{i=g+1}^m a_i\cdot \ve_i\in H_1(G,R)$ if and only if $a_i=0$ for $i=1,\cdots,m$. 
\end{proof}

Note that the above argument also shows that $\{1_R\otimes C^{\ab}_i|i=1,\cdots,g\}$ is a basis of $H_1(G,R)$. In this paper, we are particularly interested in the case $R=\mbbZ/t\mbbZ$ for $t\geq 2$. For simplicity of notation, we also write $M_t(S) := M_{\mbbZ/t\mbbZ}(S)$. Then $H_1(G,\mbbZ/t\mbbZ) \simeq M_t(\vE_\mfrako^c(T) )\simeq(\mbbZ/t\mbbZ)^g$. 

\begin{remark} \label{R:caset2}
The case $t=2$ is special. In this case, the orientation does not matter, since we always have  $1\cdot \ve = 1\cdot \ve^{-1}$ as elements of $C_1(G,\mbbZ/2\mbbZ)$ for all oriented edges $\ve\in\vE(G)$.  Hence, we can simply write $M_2(F)=M_2(\vF)$ where $\vF$ is a subset of $\vE_\mfrako(G)$ and $F\subseteq E(G)$ is the set of underlying edges corresponding to the oriented edges in $\vF$. In particular, $M_2(E^c(T))=M_2(\vE_\mfrako^c(T))$ as in the statement of Theorem~\ref{T:main}.
\end{remark}

\subsection{Twisted vertex and edge adjacency matrices of graphs}

As usual, the \emph{vertex adjacency matrix} $A=(a_{ij})_{n\times n}$ of $G$ is an $n\times n$ matrix defined as follows: for each pair of vertices $v_i,v_j\in V$, $a_{ij}$ is the number of oriented edges with initial vertex $v_i$ and terminal vertex $v_j$. In particular, if $G$ is a simple graph, then its adjacency matrix is a 0-1 matrix with $a_{ij}=1$ if $v_i$ and $v_j$ are adjacent, and $a_{ij}=0$ otherwise. 

Analogously, we define the adjacency matrix for the oriented edges of $G$. For two oriented edges $\ve,\ve' \in\vE$, we say \emph{$\ve$ feeds into $\ve'$} if $\ve'\neq \ve^{-1}$ and the terminal vertex of $\ve$ is the same as the initial vertex of $\ve'$, i.e., $\ve(1)=\ve'(0)$. The \emph{edge adjacency matrix} $B=(b_{ij})_{2m\times 2m}$ of $G$ is the $2m\times 2m$ matrix defined as follows\footnote{This nonbacktracking edge adjacency matrix is often called the Hashimoto matrix; see, for example, \cite{T2011Zeta}.}: for each pair of oriented edges $\ve_i,\ve_j\in \vE$, if $\ve_i$ feeds into $\ve_j$, then $b_{ij}=1$, and otherwise $b_{ij}=0$. 

Now we will show how to ``twist'' the vertex and edge adjacency matrices by introducing phase shifts to the edges. 

\begin{definition} \label{D:twist}
 For $\alpha=a_1\cdot \ve_1+\cdots a_m\cdot \ve_m$ and $\gamma = \gamma_1\cdot \ve_1+\cdots \gamma_m\cdot \ve_m$ in $M_t(\vE_\mfrako(G))=C_1(G,\mbbZ/t\mbbZ)$, let $\chi_\gamma(\alpha):=\epsilon_t^{\langle\gamma,\alpha\rangle}=\epsilon_t^{\gamma_1a_1+\cdots+\gamma_ma_m}$. 

\begin{enumerate}[(i)]
\item The \emph{twisted vertex adjacency matrix} $A_\gamma=(a_{ij})_{n\times n}$ of  $G$ with respect to $\gamma$ is an $n\times n$ matrix defined as follows: 
 for each pair of (not necessarily distinct)  vertices $v_i,v_j\in V(G)$, $a_{ij}=\sum_{\ve\in \vE(G),\ve(0)=v_i,\ve(1)=v_j} \chi_\gamma(1\cdot \ve)$. 

\item The \emph{twisted edge adjacency matrix} $B_\gamma=(b_{ij})_{2m\times 2m}$ of  a graph $G$ with respect to $\gamma$ is a $2m\times 2m$ matrix defined as follows: for each pair of oriented edges $\ve_i,\ve_j\in \vE$, if $\ve_i$ feeds into $\ve_j$, then $b_{ij} = \epsilon_t^{\gamma_i}$ if $\ve_i\in \vE_\mfrako$ and $b_{ij} = \epsilon_t^{-\gamma_i}$ if $\ve_i^{-1}\in \vE_\mfrako$; otherwise $b_{ij}=0$.
\end{enumerate}

\end{definition}

Clearly the trivial twisted vertex and edge adjacency matrices  (meaning $\gamma=0$) are the original vertex and edge adjacency matrices respectively. In addition, one can easily verify that twisted vertex adjacency matrices $A_\gamma$ are always Hermitian, which means that all the eigenvalues of $A_\gamma$ are real. This is not true for the twisted edge adjacency matrices. However, it can be shown that all the non-real eigenvalues of $B_\gamma$ are composed of conjugate pairs. 

\subsection{BEST theorem} 
An \emph{arborescence} rooted at vertex $w\in V(G)$ compatible with an orientation $\mfrako$ of $G$ is an oriented spanning tree $T_\mfrako$ of $G$ such that there is a unique directed path on $T_\mfrako$ from any vertex $v\in V(G)$ to $w$. 

If $\mfrako$ is an Eulerian orientation on the Eulerian graph $G$, then we say $G_\mfrako$ is an Eulerian digraph. Equivalently, this means that the in-degree is identical to the out-degree at all vertices of  $G_\mfrako$. The following well-known theorem, usually called the BEST theorem, is due to de Bruijn and van Aardenne-Ehrenfest~\cite{AB1951circuits} with a special case earlier shown by Smith and Tutte~\cite{ST1941unicursal}. 

\begin{theorem}[BEST theorem]\label{T:best}
Let $G$ be an Eulerian graph and $\mfrako$ an Eulerian orientation on $G$. Let $\In_w(G,\mfrako)$ be the number of arborescences rooted at $w\in V(G)$ compatible with $\mfrako$. Then $\In_w(G,\mfrako)$ is independent of the root vertex $w$ and 
$$\ec(G,\mfrako)=\In_w(G,\mfrako)\cdot \prod_{v\in V(G)}(\outdeg_\mfrako(v)-1)!$$ 
where $\outdeg_\mfrako(v)$ is the out-degree of $G_\mfrako$ at $v\in V(G)$. 
\end{theorem}

The BEST theorem provides an efficient method of counting Eulerian cycles on an Eulerian digraph $G_\mfrako$ by relating it to the number of arborescences on $G_\mfrako$, which can be computed using a directed version of Kirchhoff's matrix-tree theorem (see Example~\ref{E:eulerian}). Now let $\mathcal{EO}(G)$ be the set of Eulerian orientations of $G$. Then we have a general method of computing $\ec(G)$  using the identity $\ec(G)=\sum_{\mfrako\in \mathcal{EO}(G)}\ec(G,\mfrako)$.

\section{Counting circuits in homology classes}
The primary counting tools we utilize to prove the main theorem (Theorem~\ref{T:main}) and its directed version (Theorem~\ref{T:directed_eulerian}) are the trace formulas stated in Theorem~\ref{T:counting0} and Theorem~\ref{T:counting} in this section. These trace formulas can be viewed as extensions of the unweighted finite-quotient Fourier formulas in Theorem~3.21(b) of \cite{LR2024spectral}, together with the canonical isomorphism $H_1(G,\mbbZ/t\mbbZ)\simeq H_1(G,\mbbZ)/t H_1(G,\mbbZ)$ in Lemma~2.5(c) there. We emphasize two differences in context: (i) the formulas in this section are stated for both twisted vertex and edge adjacency matrices, whereas Theorem~3.21(b) of \cite{LR2024spectral} uses the twisted edge adjacency matrix; and (ii) $M_t(\vF)$ in Theorem~\ref{T:counting0} need not be a quotient of $H_1(G,\mbbZ)$, as required there. For example, $M_t(\vE_\mfrako(G))=C_1(G,\mbbZ/t\mbbZ)$ is generally not a quotient of $H_1(G,\mbbZ)$.
Moreover, the proofs in this section are self-contained, not relying on the comprehensive framework established in \cite{LR2024spectral}. 

We first define counting functions based on two types of submodules of the $\mbbZ/t\mbbZ$-module $C_1(G,\mbbZ/t\mbbZ) =M_t(\vE_\mfrako(G) )$, i.e., $M_t(\vF)$ for some $\vF\subseteq \vE_\mfrako(G) $ and the first homology group $H_1(G,\mbbZ/t\mbbZ)$. 

For $\alpha = \sum_{\ve\in \vE_\mfrako(G)} a_\ve\cdot \ve \in C_1(G,\mbbZ)$, we will use the notation $\alpha_{[t]}:=\sum_{\ve\in \vE_\mfrako(G)} a'_\ve\cdot \ve \in C_1(G,\mbbZ/t\mbbZ) $ where $a'_\ve=a_\ve \mod t$. 
\begin{definition} \label{D:count-function}
We define the following counting functions of circuits:
\begin{enumerate}[(i)]
\item For a subset $\vF$ of $\vE_\mfrako(G)$ and $\alpha\in M_t(\vF)$, $N_{t,\vF}(\alpha,l)$ (respectively $\widetilde{N}_{t,\vF}(\alpha,l)$) is defined to be the number of circuits (respectively closed walks) $C$ of length $l$ such that $C^{\ab}_{[t]}|_\vF=\alpha$.
\item For a special case $\vF=\vE_\mfrako(G)$ of (i), we also write $N_t(\alpha,l):=N_{t,\vF}(\alpha,l)$ and $\widetilde{N}_t(\alpha,l)=\widetilde{N}_{t,\vF}(\alpha,l)$ for each $\alpha \in M_t(\vF)=C_1(G,\mbbZ/t\mbbZ)$. 
\end{enumerate}
\end{definition}

A quick result about Definition~\ref{D:count-function}(ii) is shown in the following lemma, which is simply a restatement of the fact that $C^{\ab}\in H_1(G,\mbbZ)$ for all $C\in \widetilde{\mcalC}(G)$. 
\begin{lemma}
For all $\alpha\in C_1(G,\mbbZ/t\mbbZ) \setminus H_1(G,\mbbZ/t\mbbZ)$ and any length $l$, $N_t(\alpha,l)=\widetilde{N}_t(\alpha,l)=0$. 
\end{lemma}

The main results of this section are stated in the following two theorems: Theorem~\ref{T:counting0} provides a trace formula for $N_{t,\vF}(\alpha,l)$ and $\widetilde{N}_{t,\vF}(\alpha,l)$ with $\alpha\in M_t(\vF)$ (Definition~\ref{D:count-function}(i)), and Theorem~\ref{T:counting} provides a trace formula for $N_t(\alpha,l)$ and $\widetilde{N}_t(\alpha,l)$ with $\alpha\in H_1(G,\mbbZ/t\mbbZ)$ (Definition~\ref{D:count-function}(ii)). 

\begin{theorem}\label{T:counting0}
For a subset $\vF$ of $\vE_\mfrako(G)$ and $\alpha\in M_t(\vF)$,  we have
$$N_{t,\vF}(\alpha,l) = \frac{1}{t^{|\vF|}}\sum_{\gamma\in M_t(\vF)}\epsilon_t^{-\langle\gamma, \alpha\rangle}\trace\left(B_\gamma^l\right)$$
and
$$ \widetilde{N}_{t,\vF}(\alpha,l)= \frac{1}{t^{|\vF|}}\sum_{\gamma\in M_t(\vF)}\epsilon_t^{-\langle\gamma, \alpha\rangle}\trace\left(A_\gamma^l\right)$$
for any length $l$. 
\end{theorem}

\begin{theorem}\label{T:counting}
For a graph $G$ of genus $g$, a spanning tree $T$ of  $G$ and $\alpha\in H_1(G,\mbbZ/t\mbbZ)$, we have
$$N_t(\alpha,l) = \frac{1}{t^g}\sum_{\gamma\in M_t(\vE_\mfrako^c(T))}\epsilon_t^{-\langle\gamma, \alpha\rangle}\trace\left(B_\gamma^l\right)$$
and
$$ \widetilde{N}_t(\alpha,l)= \frac{1}{t^g}\sum_{\gamma\in M_t(\vE_\mfrako^c(T))}\epsilon_t^{-\langle\gamma, \alpha\rangle}\trace\left(A_\gamma^l\right)$$
for any length $l$. 
\end{theorem}

\begin{remark}
For $t=2$, the character sums in Theorem~\ref{T:counting} are ordinary Fourier transforms on the elementary abelian $2$-group $M_2(E^c(T))$.  Thus, after the trace values $\trace(B_\gamma^l)$ or $\trace(A_\gamma^l)$ have been computed for all $\gamma$, the counts in all homology classes can be recovered together by applying the inverse Fast Walsh--Hadamard transform.  This post-processing step costs only $O(g2^g)$ arithmetic operations and does not change the dominant cost, which is the computation of the $2^g$ trace values.
\end{remark}

Before proving these theorems, we will first present a counting lemma (Lemma~\ref{L:circuit-counting}) which provides a trace characterization of  the counting problem of circuits and closed paths, and then show how the twisting of the adjacency matrices affects the trace functions. 

Again, let $\mcalC_l$ (respectively $\widetilde{\mcalC}_l$) be the set of all circuits (respectively all closed walks) of length $l$ on $G$. Then the following lemma states a standard fact, saying that the cardinality of $\mcalC_l$ (respectively $\widetilde{\mcalC}_l$) is essentially the trace of the $l$-th power of the edge (respectively vertex) adjacency matrix of $G$. For a square matrix $M$, denote by $\spec M$ the \emph{spectrum} of $M$, i.e., the multi-set of all eigenvalues of $M$. 

\begin{lemma}\label{L:circuit-counting}
Let $A$ and $B$ be the vertex and edge adjacency matrices of graph $G$ respectively. For any length $l$, we have
\begin{enumerate}[(a)]
\item $|\mcalC_l|=\trace\left(B^l\right)=\sum_{\lambda\in \spec B}\lambda^l$;
\item $|\widetilde{\mcalC}_l|=\trace\left(A^l\right)=\sum_{\lambda\in \spec A}\lambda^l$.
\end{enumerate}
\end{lemma}

\begin{proof}
Let $A=(a_{ij})_{n\times n}$ and $B=(b_{ij})_{2m\times 2m}$. 

The diagonal entries of $B^l=(b^{(l)}_{ij})_{2m\times 2m}$ can be written as $$b^{(l)}_{ii}=\sum_{k_1,\cdots,k_{l-1}}b_{ik_1}b_{k_1k_2}\cdots b_{k_{l-1}i}.$$  Note that $b_{ik_1}b_{k_1k_2}\cdots b_{k_{l-1}i}=1$ if and only if $\ve_i\ve_{k_1}\cdots\ve_{k_{l-1}}$ is a circuit; otherwise $b_{ik_1}b_{k_1k_2}\cdots b_{k_{l-1}i}=0$. Therefore, $\trace\left(B^l\right)=\sum_{i=1}^{2m}b^{(l)}_{ii}$ exactly counts the number of all circuits of length $l$ on $G$, i.e., $\trace\left(B^l\right) = |\mcalC_l|$.

The diagonal entries of  $A^l=(a^{(l)}_{ij})_{n\times n}$ can be written as $$a^{(l)}_{ii}=\sum_{k_1,\cdots,k_{l-1}}a_{ik_1}a_{k_1k_2}\cdots a_{k_{l-1}i}.$$
Note that by definition of vertex adjacency matrix, $a_{ik_1}a_{k_1k_2}\cdots a_{k_{l-1}i}$ is the number of closed walks of length $l$ going through the vertices $v_i,v_{k_1},\cdots,v_{k_{l-1}},v_i$ consecutively. Therefore, $\trace\left(A^l\right) = \sum_{v\in V}a^{(l)}_{vv}$ exactly counts the number of all closed walks of length $l$ on $G$, i.e., $\trace\left(A^l\right) = |\widetilde{\mcalC}_l|$. Analogously, we can show that $\trace\left(B^l\right) = |\mcalC_l|$. Note that in general $\trace\left(A^l\right) \geq \trace\left(B^l\right)$, since the former counts all the closed walks (allowing backtracks and tail) while the latter only counts the circuits, i.e., the closed paths without backtracks and tail. 

It remains to show that $\trace\left(A^l\right)=\sum_{\lambda\in \spec A}\lambda^l$ and $\trace\left(B^l\right)=\sum_{\lambda\in \spec B}\lambda^l$. But this follows from a standard fact in linear algebra, i.e., $\trace(M^l)=\sum_{\lambda\in \spec M}\lambda^l$ for a square matrix $M$ (a simple proof can be given based on the Jordan canonical form of $M$). 
\end{proof}

 For $\vF\subseteq \vE_\mfrako$, let $M_t(\vF)$ be naturally identified as a submodule of $M_t(\vE_\mfrako)=C_1(G,\mbbZ/t\mbbZ)$. For each $\gamma\in M_t(\vF)$ and $C\in\widetilde{\mcalC}$, we let $\chi_\gamma(C) :=\chi_\gamma(C^{\ab})=\epsilon_t^{\langle\gamma, C^{\ab}_{[t]}\rangle}=\epsilon_t^{\langle\gamma, C^{\ab}_{[t]}|_\vF\rangle}$. The following lemma is a ``twisted'' version of Lemma~\ref{L:circuit-counting}.

\begin{lemma} \label{L:trace}
We have the following identities:
\begin{enumerate}[(a)]
\item $\trace\left(B_\gamma^l\right)$ is real, and $\trace\left(B_\gamma^l\right)=\sum_{C\in \mcalC_l}\chi_\gamma(C)=\sum_{\lambda\in \spec B_\gamma}\lambda^l$;
\item $\trace\left(A_\gamma^l\right)$ is real, and $\trace\left(A_\gamma^l\right)=\sum_{\widetilde{\mcalC}_l}\chi_\gamma(C)=\sum_{\lambda\in \spec A_\gamma}\lambda^l$;
\end{enumerate}
\end{lemma}

\begin{proof}
Let $A_\gamma=(a_{ij})_{n\times n}$ and $B_\gamma=(b_{ij})_{2m\times 2m}$. 

As in the proof of Lemma~\ref{L:circuit-counting}, the diagonal entries of $B_\gamma^l=\left(b^{(l)}_{ij}\right)_{2m\times 2m}$ can be written as $$b^{(l)}_{ii}=\sum_{k_1,\cdots,k_{l-1}}b_{ik_1}b_{k_1k_2}\cdots b_{k_{l-1}i}.$$  Again $b_{ik_1}b_{k_1k_2}\cdots b_{k_{l-1}i}$ is nonzero if and only if $\ve_i\ve_{k_1}\cdots\ve_{k_{l-1}}$ is a circuit.  However, the contribution from such a circuit is not necessarily $1$ as in Lemma~\ref{L:circuit-counting}.  Actually, by the definition of twisted edge adjacency matrix (Definition~\ref{D:twist}), one may verify that the contribution of $C$ is exactly $\chi_\gamma(C)$.  
Therefore, $\trace\left(B_\gamma^l\right)=\sum_{C\in \mcalC_l}\chi_\gamma(C)$.

 The diagonal entries of $A_\gamma^l$ can be written as $$a^{(l)}_{vv}=\sum_{u_1,\cdots,u_{l-1}\in V}a_{vu_1}a_{u_1u_2}\cdots a_{u_{l-1}v}.$$ 
 Analogously, $a^{(l)}_{vv}$ can be interpreted as the sum of $\chi(C)$ with $C$ running through all closed paths with base vertex $v$.  Therefore, we have $\trace\left(A_\gamma^l\right)=\sum_{\widetilde{\mcalC}_l}\chi_\gamma(C)$. 

The identities $\trace\left(A_\gamma^l\right)=\sum_{\lambda\in \spec A_\gamma}\lambda^l$ and $\trace\left(B_\gamma^l\right)=\sum_{\lambda\in \spec B_\gamma}\lambda^l$ also follow from the standard linear algebraic fact that $\trace(M^l)=\sum_{\lambda\in \spec M}\lambda^l$ for a square matrix $M$. 

It remains to show that $\trace\left(A_\gamma^l\right)$ and $\trace\left(B_\gamma^l\right)$  are all reals. Note that the eigenvalues of $A_\gamma$ are all reals since $A_\gamma$ is Hermitian. This implies that $\trace\left(A_\gamma^l\right)$ must also be real for all $l$. For $\trace\left(B_\gamma^l\right)$, we note that circuits always come in opposite pairs, i.e., if  $C$ is a circuit, then reversing the traveling direction of $C$, we always get another circuit $C^{-1}$ of the same length. Moreover, it is straightforward to verify that $\chi_\gamma(C^{-1})$ is always the complex conjugate of $\chi_\gamma(C)$. Therefore, the identity $\trace\left(B_\gamma^l\right)=\sum_{C\in \mcalC_l}\chi_\gamma(C)$ shows that $\trace\left(B_\gamma^l\right)$ must always be real. 
\end{proof}

\begin{lemma} \label{L:orth}
For all $\vF\subseteq \vE_\mfrako(G)$, $\alpha\in M_t(\vF)$ and $C\in\widetilde{\mcalC}$, we have the following orthogonal relation:
$$\sum_{\gamma\in M_t(\vF)}\epsilon_t^{-\langle\gamma, \alpha\rangle}\chi_\gamma(C)  =\begin{cases}
      t^{|\vF|}, & \text{if}\ C^{\ab}_{[t]}|_\vF=\alpha; \\
      0, & \text{otherwise}.
    \end{cases}$$
\end{lemma}

\begin{proof}
By definition, $\chi_\gamma(C) =\epsilon_t^{\langle\gamma, C^{\ab}_{[t]}\rangle}=\epsilon_t^{\langle\gamma, C^{\ab}_{[t]}|_\vF\rangle}$. Therefore, 
$$\sum_{\gamma\in M_t(\vF)}\epsilon_t^{-\langle\gamma, \alpha\rangle}\chi_\gamma(C)=\sum_{\gamma\in M_t(\vF)}\epsilon_t^{\langle\gamma, C^{\ab}_{[t]}|_\vF-\alpha\rangle}.$$ 

Clearly if $C^{\ab}_{[t]}|_\vF=\alpha$, then $\sum_{\gamma\in M_t(\vF)}\epsilon_t^{\langle\gamma, C^{\ab}_{[t]}|_\vF-\alpha\rangle}=|M_t(\vF)|=t^{|\vF|}$. 

Now suppose  $C^{\ab}_{[t]}|_\vF\neq\alpha$. Without loss of generality, we may assume $\vF=\{\ve_1,\cdots, \ve_f\}$ where $f\leq m$. Let $\alpha=a_1\cdot \ve_1+\cdots + a_f\cdot \ve_f$ and $C^{\ab}_{[t]}|_\vF=b_1\cdot \ve_1+\cdots + b_f\cdot \ve_f$.  We may also assume there exists $1\leq k \leq f$ such that $a_i\neq b_i$ for $1\leq i\leq k$ and $a_i = b_i$ for $k+1\leq i\leq f$. This implies  $ C^{\ab}_{[t]}|_\vF-\alpha=c_1 \cdot \ve_1+\cdots +c_k\cdot \ve_k$ where $c_i=b_i-a_i\neq 0$ for $1\leq i\leq k$. 

For $\gamma = \gamma_1\cdot \ve_1+\cdots \gamma_f\cdot \ve_f\in M_t(\vF)$, we have 
$\langle\gamma, C^{\ab}_{[t]}|_\vF-\alpha\rangle = \sum_{i=1}^k \gamma_ic_i$. Note that $\sum_{\gamma_i=0}^{t-1}\epsilon_t^{\gamma_ic_i}=0$ for each $1\leq i \leq k$,  since $\epsilon_t^{c_i}$ is a primitive $t/\gcd(t,c_i)$-th root of unity. This means
\begin{align*}
\sum_{\gamma\in M_t(\vF)}\epsilon_t^{\langle\gamma, C^{\ab}_{[t]}|_\vF-\alpha\rangle}&=\sum_{\gamma\in M_t(\vF)}\epsilon_t^{\sum_{i=1}^k \gamma_ic_i} \\
&=t^{f-k}\left(\sum_{\gamma_1=0}^{t-1}\epsilon_t^{\gamma_1c_1}\right)\cdots \left(\sum_{\gamma_k=0}^{t-1}\epsilon_t^{\gamma_kc_k}\right)\\
&=0, 
\end{align*}
and the claim is proved. 
\end{proof}

Now we are ready to prove Theorem~\ref{T:counting0} and Theorem~\ref{T:counting} . 

\begin{proof}[Proof of Theorem~\ref{T:counting0}]
Then by Lemma~\ref{L:trace}(a) and Lemma~\ref{L:orth}, we have
\begin{align*}
& \frac{1}{t^{|\vF|}}\sum_{\gamma\in M_t(\vF)}\epsilon_t^{-\langle\gamma, \alpha\rangle}\trace\left(B_\gamma^l\right) \\
=& \frac{1}{t^{|\vF|}}\sum_{\gamma\in M_t(\vF)}\epsilon_t^{-\langle\gamma, \alpha\rangle}\sum_{C\in \mcalC_l}\chi_\gamma(C) \\
=& \frac{1}{t^{|\vF|}}\sum_{C\in \mcalC_l} \sum_{\gamma\in M_t(\vF)}\epsilon_t^{-\langle\gamma, \alpha\rangle}\chi_\gamma(C)  \\
=& \frac{1}{t^{|\vF|}}\sum_{C\in \mcalC_l, C^{\ab}_{[t]}|_{\vF}=\alpha}   t^{|\vF|} = N_{t,\vF}(\alpha,l).
\end{align*}

Analogously, by Lemma~\ref{L:trace}(b), we have
\begin{align*}
 & \frac{1}{t^{|\vF|}}\sum_{\gamma\in M_t(\vF)}\epsilon_t^{-\langle\gamma, \alpha\rangle}\trace\left(A_\gamma^l\right) \\
=& \frac{1}{t^{|\vF|}}\sum_{\gamma\in M_t(\vF)}\epsilon_t^{-\langle\gamma, \alpha\rangle}\sum_{C\in \widetilde{\mcalC}_l}\chi_\gamma(C) \\
=& \frac{1}{t^{|\vF|}}\sum_{C\in \widetilde{\mcalC}_l} \sum_{\gamma\in M_t(\vF)}\epsilon_t^{-\langle\gamma, \alpha\rangle}\chi_\gamma(C)  \\
=&  \frac{1}{t^{|\vF|}}\sum_{C\in \widetilde{\mcalC}_l, C^{\ab}_{[t]}|_{\vF}=\alpha}  t^{|\vF|} = \widetilde{N}_{t,\vF}(\alpha,l).
\end{align*}

\end{proof}

\begin{proof}[Proof of Theorem~\ref{T:counting}]
As shown in Lemma~\ref{L:rhoT}, $H_1(G,\mbbZ/t\mbbZ)$ is isomorphic to $M_t(\vE_\mfrako^c(T))$ by the restriction map $\alpha \mapsto \alpha|_{\vE_\mfrako^c(T)}$. This means that for all  $\gamma\in M_t(\vE_\mfrako^c(T))$ and $\alpha \in H_1(G,\mbbZ)$, we have $ \langle\gamma, \alpha\rangle = \langle\gamma, \alpha|_{\vE_\mfrako^c(T)}\rangle$, 
$N_t(\alpha,l)  = N_{t,\vE_\mfrako^c(T)}(\alpha|_{\vE_\mfrako^c(T)},l) $ and $ \widetilde{N}_t(\alpha,l)  =  \widetilde{N}_{t,\vE_\mfrako^c(T)}(\alpha|_{\vE_\mfrako^c(T)},l)$. 
Then Theorem~\ref{T:counting} is implied by Theorem~\ref{T:counting0}. 
\end{proof}

\section{Proof of the main theorem and its directed version}

By Remark~\ref{R:caset2}, when considering $C_1(G,\mbbZ/2\mbbZ)$, the orientation $\mfrako$ does not matter, and we write $M_2(E^c(T))=M_2(\vE_\mfrako^c(T))$ for a spanning tree $T$ of $G$. Now let $\mathbbm{1}= 1\cdot e_1+\cdots+1\cdot e_m \in C_1(G,\mbbZ/2\mbbZ)$ and $\mathbbm{1}_T= 1\cdot e_1+\cdots+1\cdot e_g \in M_2(E^c(T))$ where $E^c(T)=\{e_1,\cdots,e_g\}$. As shown in Lemma~\ref{L:rhoT},  let $\rho_T:H_1(G,\mbbZ/2\mbbZ) \xrightarrow{\sim} M_2(E^c(T)) $ be the restriction map which is an isomorphism. 

\begin{proposition} \label{P:eulerian}
The following statements are equivalent:
\begin{enumerate}[(a)]
\item  $G$ is an Eulerian graph;
\item The degree of every vertex of $G$ is even;
\item $\mathbbm{1}\in H_1(G,\mbbZ/2\mbbZ)$;
\item $\rho_T^{-1}(\mathbbm{1}_T)=\mathbbm{1}$ for all spanning trees $T$ of $G$. 
\item $G$ is bridgeless (i.e., $G\setminus e$ is still connected for all $e\in E$) and $\rho_T^{-1}(\mathbbm{1}_T)$ is identical for all spanning trees $T$ of $G$. 
\end{enumerate}
\end{proposition}

\begin{proof}
The equivalence of (a) and (b) is actually the statement of the Euler's Theorem. 

Now suppose for each vertex $v\in V$, the degree of $v$ is even. Equivalently, this means $\sum_{e\in E\text{ and } e \text{ is adjacent to } v} 1 = 0 \mod 2$ for all $v\in V$, i.e., the balancing condition for $\mathbbm{1}$ to be a $\mbbZ/2\mbbZ$-circulation is satisfied. Hence we have the equivalence of (b) and (c). 

By Lemma~\ref{L:rhoT}, the restriction map $\rho_T$ is an isomorphism. Then if $\mathbbm{1}\in H_1(G,\mbbZ/2\mbbZ)$, then $\rho_T(\mathbbm{1})=\mathbbm{1}_T$ and hence $\rho_T^{-1}(\mathbbm{1}_T)=\mathbbm{1}$. This means that (c) implies (d), and the converse is obviously true. 

Since Eulerian graphs are clearly bridgeless, it is straightforward to see that (d) implies (e). Now suppose (e) is satisfied. Let $\alpha$ be the identical inverse image of $\mathbbm{1}_T$ for all spanning trees of $G$. Note that an edge $e\in E$ is a bridge if and only if $e$ appears in all spanning trees of $G$. Since $G$ is bridgeless, each edge $e$ must appear in some spanning trees  $T$ of $G$. By the nature of the restriction map $\rho_T$, this means that the coefficient of $e$ in $\alpha$ must be $1$. As a result, we see $\alpha = \mathbbm{1}$ and (d) holds. 

\end{proof}

\begin{lemma} \label{L:eulerian}
Let $G$ be a graph with $m$ edges. Then $\ec(G)=N_2(\mathbbm{1},m)/m = \widetilde{N}_2(\mathbbm{1},m)/m$
\end{lemma}

\begin{proof}
It follows from the fact that a closed walk $C$ on $G$ is an Eulerian circuit if and only if $\mathbbm{1}=C^{\ab}_{[2]}\in H_1(G,\mbbZ/2\mbbZ)$ and the length of $C$ is $m$.
\end{proof}

Now let us prove the main theorem.

\begin{proof}[Proof of Theorem~\ref{T:main}]
The above lemma gives $\ec(G)=N_2(\mathbbm{1},m)/m = \widetilde{N}_2(\mathbbm{1},m)/m$. Since $G$ is Eulerian by assumption,  $\mathbbm{1}=\rho_T^{-1}(\mathbbm{1}_T)$ for all spanning trees $T$ of $G$ by Proposition~\ref{P:eulerian}. 

Using the formula in Theorem~\ref{T:counting}, we have
\begin{align*}
N_2(\mathbbm{1},m) = N_2(\rho_T^{-1}(\mathbbm{1}_T),m) &= \frac{1}{2^g}\sum_{\gamma\in M_2(E^c(T))}(-1)^{-\langle\gamma, \mathbbm{1}\rangle}\trace\left(B_\gamma^m\right) \\
&= \frac{1}{2^g}\sum_{\gamma\in M_2(E^c(T))}(-1)^{-\langle\gamma, \mathbbm{1}_T\rangle}\trace\left(B_\gamma^m\right) \\
&=\frac{1}{2^g}\sum_{\gamma\in M_2(E^c(T))}(-1)^{\sigma(\gamma)}\trace\left(B_\gamma^m\right),
\end{align*} 
and 
\begin{align*}
\widetilde{N}_2(\mathbbm{1},m)= \widetilde{N}_2(\rho_T^{-1}(\mathbbm{1}_T),m)  &= \frac{1}{2^g}\sum_{\gamma\in M_2(E^c(T))}(-1)^{-\langle\gamma, \mathbbm{1}\rangle}\trace\left(A_\gamma^m\right) \\
&= \frac{1}{2^g}\sum_{\gamma\in M_2(E^c(T))}(-1)^{-\langle\gamma, \mathbbm{1}_T\rangle}\trace\left(A_\gamma^m\right) \\
&=\frac{1}{2^g}\sum_{\gamma\in M_2(E^c(T))}(-1)^{\sigma(\gamma)}\trace\left(A_\gamma^m\right).
\end{align*}

\end{proof}

\begin{remark} \label{R:nonEulerian}
If $G$ is not Eulerian, then Theorem~\ref{T:main} is not true in general. Actually, in this case, we have $\ec(G)=N_2(\mathbbm{1},m)=\widetilde{N}_2(\mathbbm{1},m)=0$. But for each spanning tree $T$, we may still compute $N_2(\rho_T^{-1}(\mathbbm{1}_T),m) $ and $\widetilde{N}_2(\rho_T^{-1}(\mathbbm{1}_T),m) $ which are not necessarily zero. The reason is that $\mathbbm{1}\notin H_1(G,\mbbZ/2\mbbZ)$ (Proposition~\ref{P:eulerian}) and the identity $\rho_T^{-1}(\mathbbm{1}_T)=\mathbbm{1}$ does not hold anymore. See Example~\ref{E:non-eulerian} for an example in such a case. 
\end{remark}

Now let $G_\mfrako$ be an Eulerian digraph, which means $\mfrako$ is an Eulerian orientation on the Eulerian graph $G$. We can also use Theorem~\ref{T:counting} to count the number $\ec(G,\mfrako)$ of Eulerian cycles on $G_\mfrako$.  In particular, for $t\geq 3$, each Eulerian orientation $\mfrako$ induces uniquely a circulation in $H_1(G,\mbbZ/t\mbbZ)$ in a natural way. Therefore, we may prove analogously the following theorem on an explicit formula for $\ec(G,\mfrako)$, which can be considered as a directed version of Theorem~\ref{T:main}.

\begin{theorem} \label{T:directed_eulerian}
For an Eulerian graph $G$ of genus $g$ with $m$ edges, let $T$ be a spanning tree of $G$ and $\mfrako$ an Eulerian orientation of $G$.  For each $t\geq 3$, we have the following trace formula for the number of Eulerian cycles on $G_\mfrako$:
\begin{align*}
\ec(G,\mfrako) &= \frac{1}{m\cdot t^g}\sum_{\gamma\in M_t(\vE_\mfrako^c(T))}\epsilon_t^{-\sigma(\gamma)}\trace\left(B_\gamma^m\right) \\
&= \frac{1}{m\cdot t^g}\sum_{\gamma\in M_t(\vE_\mfrako^c(T))}\epsilon_t^{-\sigma(\gamma)}\trace\left(A_\gamma^m\right),
\end{align*}
where $\sigma(\gamma)=\sum_{\ve\in \vE_\mfrako^c(T)} c_\ve$ for $\gamma = \sum_{\ve\in \vE_\mfrako^c(T)}c_\ve \cdot \ve \in M_t(\vE_\mfrako^c(T) )$. 
\end{theorem}
\begin{proof}
Let $\vE_\mfrako^c(T)=\{\ve_1,\cdots,\ve_g\}$ and $\vE_\mfrako(T)=\{\ve_{g+1},\cdots,\ve_m\}$. Let $\mathbbm{1}_\mfrako:= 1\cdot \ve_1+\cdots+1\cdot \ve_m \in C_1(G,\mbbZ/t\mbbZ)$ and $\mathbbm{1}_{\mfrako,T}:= 1\cdot \ve_1+\cdots+1\cdot \ve_g \in M_t(\vE_\mfrako^c(T))$. 

Then a closed walk $C$ on $G$ is an Eulerian circuit compatible with the orientation $\mfrako$ (which means $\mfrako$ must be an Eulerian orientation) if and only $\mathbbm{1}_\mfrako=C^{\ab}_{[t]}\in H_1(G,\mbbZ/t\mbbZ)$ and the length of $C$ is $m$. Hence, we must have $\ec(G,\mfrako)=N_t(\mathbbm{1}_\mfrako,m) /m= \widetilde{N}_t(\mathbbm{1}_\mfrako,m)/m$.  Let $\rho_T:H_1(G,\mbbZ/t\mbbZ) \xrightarrow{\sim} M_t(\vE_\mfrako^c(T)) $ be the restriction map which is an isomorphism (Lemma~\ref{L:rhoT}). If $\mfrako$ is Eulerian, then $\mathbbm{1}_\mfrako=\rho_T^{-1}(\mathbbm{1}_{\mfrako,T})\in H_1(G,\mbbZ/t\mbbZ)$, and we can complete the proof using the formulas in Theorem~\ref{T:counting} as in the proof of Theorem~\ref{T:main}. 
\end{proof}

\section{Algorithmic consequences}\label{S:algorithmic}

We now make explicit the algorithmic content of the trace formula.  This section records the complexity interpretation suggested by the character-sum formula; it is not meant to replace specialized polynomial-time algorithms for particular graph classes such as generalized series-parallel graphs.  Because exact counting is \#P-complete in general, the purpose of the analysis is not to claim a uniform polynomial-time algorithm, but to show that the only exponential term in the general trace-sum method is the number $2^g$ of homological characters.

\begin{algorithm}[Baseline trace-sum algorithm]\label{A:baseline}
Given a connected graph $G=(V,E)$:
\begin{enumerate}[(1)]
\item Check whether all vertices have even degree.  If not, return $\ec(G)=0$.
\item Compute a spanning tree $T$ and list the cotree $E^c(T)=\{f_1,\cdots,f_g\}$.
\item For each $\gamma\in M_2(E^c(T))$, construct the twisted vertex matrix $A_\gamma$ and compute $\tau_\gamma=\trace(A_\gamma^m)$.
\item Return
$$
\frac{1}{m2^g}\sum_{\gamma\in M_2(E^c(T))}(-1)^{\sigma(\gamma)}\tau_\gamma.
$$
\end{enumerate}
\end{algorithm}

The cost of the baseline procedure is easy to read from the formula.  For a graph with $n$ vertices, $m$ edges, and genus $g$, there are $2^g$ twists to evaluate.  Constructing a twisted vertex matrix $A_\gamma$ only requires scanning the edges and inserting the corresponding signs or phases, so this part costs $O(m)$ arithmetic operations per twist.  If $\trace(A_\gamma^m)$ is obtained from an eigendecomposition of the Hermitian matrix $A_\gamma$, the dominant cost is the usual dense linear-algebra cost $O(n^3)$ per twist.  Thus the vertex-matrix implementation has the indicative cost $O(2^g n^3+2^g m)$ arithmetic operations and can store the graph and one twisted matrix at a time, using $O(n^2+m)$ working space.  The same reasoning applied directly to the edge matrices $B_\gamma$ gives the analogous, but usually larger, cost $O(2^g m^3)$ and $O(m^2)$ working space.  In this sense the trace formula is fixed-parameter tractable with respect to the genus: for fixed $g$, the remaining computation is polynomial in the graph size, while the unavoidable exponential dependence is isolated in the homological parameter.

\begin{remark}
The symmetry reductions in Theorems~\ref{T:red_anti}--\ref{T:red_counting_eulerian_combined} turn the preceding bounds into bounds with fewer trace evaluations.  Spectral antisymmetry replaces $2^g$ by $2^{g-1}$ for non-bipartite Eulerian graphs.  If a subgroup $H\leq \Aut(G)$ is used and $r_H$ denotes the number of relevant orbits of $M_2(E^c(T))$, then the orbit-reduced formula requires one trace computation per orbit representative, up to the cost of finding or supplying those representatives.
\end{remark}

\section{Examples}
We will show two examples  in this section, one for a non-Eulerian graph and one for an Eulerian graph. \footnote{Sagemath codes of the computations are available at \url{https://github.com/nittup/eulerian-count}.}

\begin{figure}
 \centering
\begin{tikzpicture}[>=to,x=2cm,y=2cm]

\tikzset{
        midarrow/.style={
            postaction={decorate},
            decoration={
                markings,
                mark=at position 0.5 with {\arrow[line width=1.5pt]{stealth'}}
            }
        }
    }

\begin{scope}[shift={(0,0)}] 
\draw (-1, 1) node {(a)};

\coordinate (v0) at (0,0);
\coordinate (v1) at ($1/3*(-1.732,-1)$);
\coordinate (v2) at ($1/3*(1.732,-1)$);
\coordinate (v3) at ($1/3*(0,1.732)$);
\coordinate (v4) at ($2/3*(-1.732,-1)$);
\coordinate (v5) at ($2/3*(1.732,-1)$);
\coordinate (v6) at ($2/3*(0,1.732)$);

\path[-,font=\scriptsize,  line width=1.5pt, black]
(v0) edge[midarrow,color=gray] node[pos=0.6,anchor=south,font=\small, color=blue]{$e_0$} (v1)
(v0) edge[midarrow,color=gray] node[pos=0.6,anchor=south,font=\small, color=blue]{$e_1$}  (v2)
(v0) edge[midarrow,color=gray] node[pos=0.5,anchor=east,font=\small, color=blue]{$e_2$}  (v3)
(v1) edge[midarrow,color=gray] node[pos=0.4,anchor=south,font=\small, color=blue]{$e_3$} (v4)
(v2) edge[midarrow,color=gray] node[pos=0.4,anchor=south,font=\small, color=blue]{$e_4$}  (v5)
(v3) edge[midarrow,color=gray] node[pos=0.4,anchor=east,font=\small, color=blue]{$e_5$}  (v6)
(v4) edge[midarrow,out=-15,in=-165] node[pos=0.5,anchor=north,font=\small, color=blue]{$e_6$}  (v5)
(v5) edge[midarrow,out=105,in=-45] node[pos=0.5,anchor=west,font=\small, color=blue]{$e_8$}  (v6)
(v6) edge[midarrow,out=-140,in=80] node[pos=0.5,anchor=east,font=\small, color=blue]{$e_7$}  (v4);

\fill [black] (v0) circle (2.5pt);
\draw (v0) node[anchor=north] {$v_0$};
\fill [black] (v1) circle (2.5pt);
\draw (v1) node[anchor=north] {$v_1$};
\fill [black] (v2) circle (2.5pt);
\draw (v2) node[anchor=north] {$v_2$};
\fill [black] (v3) circle (2.5pt);
\draw (v3) node[anchor=west] {$v_3$};
\fill [black] (v4) circle (2.5pt);
\draw (v4) node[anchor=east] {$v_4$};
\fill [black] (v5) circle (2.5pt);
\draw (v5) node[anchor=west] {$v_5$};
\fill [black] (v6) circle (2.5pt);
\draw (v6) node[anchor=south] {$v_6$};

\end{scope}

\begin{scope}[shift={(3.5,0)}] 
\draw (-1, 1) node {(b)};

\coordinate (v0) at (0,0);
\coordinate (v1) at ($2/3*(-1.732,-1)$);
\coordinate (v2) at ($2/3*(1.732,-1)$);
\coordinate (v3) at ($2/3*(0,1.732)$);

\path[-,font=\scriptsize,  line width=1.5pt, black]
(v1) edge[midarrow,out=-15,in=-165] node[pos=0.5,anchor=north,font=\small, color=blue]{$e_4$}  (v2)
(v3) edge[midarrow,out=-140,in=80] node[pos=0.5,anchor=east,font=\small, color=blue]{$e_5$} (v1)
(v3) edge[midarrow] node[pos=0.6,anchor=west,font=\small, color=blue]{$e_6$} (v1)
(v1) edge[midarrow,color=gray] node[pos=0.6,anchor=north,font=\small, color=blue]{$e_0$} (v0)
(v2) edge[midarrow,out=105,in=-45] node[pos=0.5,anchor=west,font=\small, color=blue]{$e_7$}  (v3)
(v0) edge[midarrow,out=-10,in=135] node[pos=0.5,anchor=south,font=\small, color=blue]{$e_2$} (v2)
(v2) edge[midarrow,color=gray,out=165,in=-50] node[pos=0.4,anchor=east,font=\small, color=blue]{$e_1$} (v0)
(v0) edge[midarrow,color=gray] node[pos=0.4,anchor=west,font=\small, color=blue]{$e_3$} (v3);

\fill [black] (v0) circle (2.5pt);
\draw (v0) node[anchor=south east] {$v_0$};
\fill [black] (v1) circle (2.5pt);
\draw (v1) node[anchor=east] {$v_1$};
\fill [black] (v2) circle (2.5pt);
\draw (v2) node[anchor=west] {$v_2$};
\fill [black] (v3) circle (2.5pt);
\draw (v3) node[anchor=south] {$v_3$};
\end{scope}

\end{tikzpicture}
 \caption{(a) A non-Eulerian graph $G_1$; (b) An Eulerian graph $G_2$}
 \label{F:example}
\end{figure}
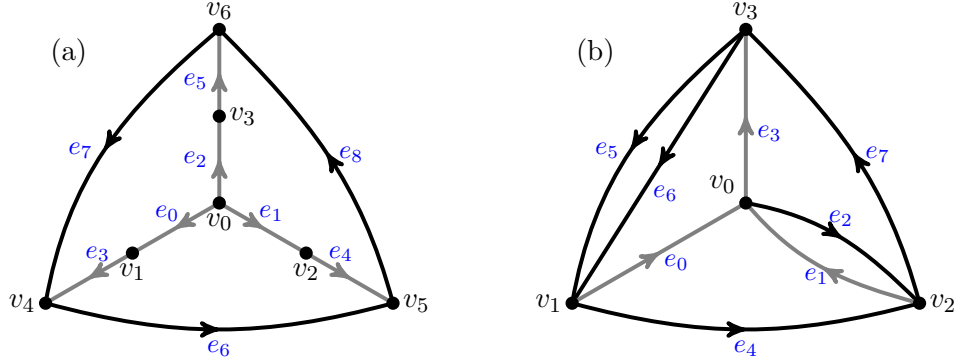

\begin{example}[The non-Eulerian case] \label{E:non-eulerian}
Let us first  consider $G_1$ with the vertex set $V(G_1)=\{v_0,\cdots,v_6\}$, the edge set $E(G_1)=\{e_0,\cdots,e_8\}$ \footnote{For convenience of comparison with codes in Sagemath, the indexing of vertices and edges begins at $0$ for the examples in this section.} and an orientation $\mfrako$ as shown in Figure~\ref{F:example}(a). Let $\vE_\mfrako(G_1)=\{\ve_0,\cdots,\ve_8\}$ be the set of positively oriented edges with respect to $\mfrako$. Clearly $G_1$ is not an Eulerian graph, i.e., $\ec(G_1)=0$, since none of the vertices have even degrees. Let $\mathbbm{1}= 1\cdot e_0+\cdots+1\cdot e_8 \in C_1(G_1,\mbbZ/2\mbbZ)=M_2(E(G_1))$. 

Let $T$ be a spanning tree of $G_1$ with edges $e_0,\cdots,e_5$. Let  $\mathbbm{1}_T= 1\cdot e_6+1\cdot e_7+1\cdot e_8 \in M_2(E^c(T))$. Then we may still compute the right hand side of the formula in Theorem~\ref{T:main} as 
$N_2(\rho_T^{-1}(\mathbbm{1}_T),9)/9=N_{2,E^c(T)}(\mathbbm{1}_T,9)/9$ and $\widetilde{N}_2(\rho_T^{-1}(\mathbbm{1}_T),9)/9=\widetilde{N}_{2,E^c(T)}(\mathbbm{1}_T,9)/9$. By computation, we derive $N_2(\rho_T^{-1}(\mathbbm{1}_T),9)= 6$ and $\widetilde{N}_2(\rho_T^{-1}(\mathbbm{1}_T),9)=3282$, both being nonzero. The reason is that in this example $\rho_T^{-1}(\mathbbm{1}_T)=\mathbbm{1}_T\neq\mathbbm{1}$ (Remark~\ref{R:nonEulerian}). Actually the 6 circuits counted by $N_2(\rho_T^{-1}(\mathbbm{1}_T),9)$ are the two opposite circuits, $\ve_6\ve_8\ve_7\ve_6\ve_8\ve_7\ve_6\ve_8\ve_7$ and $\ve_6^{-1}\ve_7^{-1}\ve_8^{-1}\ve_6^{-1}\ve_7^{-1}\ve_8^{-1}\ve_6^{-1}\ve_7^{-1}\ve_8^{-1}$, and their other 4 translations. Note that there are a lot more closed walks counted by $\widetilde{N}_2(\rho_T^{-1}(\mathbbm{1}_T),9)$ since for closed walks, backtracks are allowed. For instance, $\ve_6\ve_6^{-1}\ve_6\ve_7\ve_7^{-1}\ve_7\ve_8\ve_8^{-1}\ve_8$ is a closed walk counted by $\widetilde{N}_2(\rho_T^{-1}(\mathbbm{1}_T),9)$, but not a circuit counted by $N_2(\rho_T^{-1}(\mathbbm{1}_T),9)$. 

On the other hand, by Lemma~\ref{L:eulerian}, we may still use the formula in Theorem~\ref{T:counting0} to precisely count the number of Eulerian circuits with an expected value of $0$, by simply letting $F=E(G_1)$. As can be confirmed by computation, we get $\ec(G_1)=N_{2,E(G_1)}(\mathbbm{1},9)/9=\widetilde{N}_{2,E(G_1)}(\mathbbm{1},9)/9=0$. 
\end{example}

\begin{example}[The Eulerian case] \label{E:eulerian}
Now we will  consider an Eulerian graph $G_2$ with the vertex set $V(G_2)=\{v_0,\cdots,v_3\}$, the edge set $E(G_2)=\{e_0,\cdots,e_7\}$ and an Eulerian orientation $\mfrako$ as shown in Figure~\ref{F:example}(b).  Let $\vE_\mfrako(G_2)=\{\ve_0,\cdots,\ve_7\}$ be the set of positively oriented edges with respect to $\mfrako$. Let $T$ be a spanning tree of $G_2$ with edges $e_0,e_1,e_3$. 

 The out-degree of $G_2$ with respect to $\mfrako$ at all vertices is $2$. Hence by the BEST theorem (Theorem~\ref{T:best}), $\ec(G_2,\mfrako)=\In_w(G_2,\mfrako)$ where $\In_w(G,\mfrako)$ is the number of arborescences rooted at an arbitrary vertex $w\in V(G_2)$, which equals the determinant of the reduced Laplacian of $G_2$ under orientation $\mfrako$.  In particular, the Laplacian $L_\mfrako$ of $G_2$ under orientation $\mfrako$ is defined as $(L_\mfrako)_{ii}=\outdeg_\mfrako(v_i)$ and  $(L_\mfrako)_{ij}$ is the negation of the number of positively oriented edges with respect to $\mfrako$ with initial vertex $v_i$ and terminal vertex $v_j$. Removing the last row and column of  $L_\mfrako$, we get the reduced Laplacian $L_\mfrako^{\red(v_3)}$. As a consequence of  the directed version of Kirchhoff's matrix-tree theorem, $\In_w(G_2,\mfrako)=\det 
\left(L_\mfrako^{\red(v_3)}\right)$. Specifically, we can compute
 $$\ec(G_2,\mfrako)=\In_w(G_2,\mfrako)=\det \left(L_\mfrako^{\red(v_3)}\right)=\begin{vmatrix}
2 & 0 &-1  \\
-1 & 2& -1 \\
-1 & 0 &  2
\end{vmatrix}=6.$$

On the other hand, computations based on the trace formula in Theorem~\ref{T:directed_eulerian} will provide the same result for $\ec(G_2,\mfrako)$. That is, for any $t\geq 3$,
$$\ec(G_2,\mfrako) = N_{t,\vF}(\mathbbm{1}_{\mfrako,T},8)/8=\widetilde{N}_{t,\vF}(\mathbbm{1}_{\mfrako,T},8)/8=6$$ where $\vF=\vE_\mfrako^c(T)=\{\ve_2,\ve_4,\ve_5,\ve_6,\ve_7\}$ and $\mathbbm{1}_{\mfrako,T}= 1\cdot \ve_2+1\cdot \ve_4+1\cdot \ve_5+1\cdot \ve_6+1\cdot \ve_7 \in M_t(\vF)$. 

To compute $\ec(G_2)$ using the BEST theorem, we need to find all the Eulerian orientations $\mfrako'$ of $G_2$ and get the sum of all $\ec(G,\mfrako')$, i.e., 
$\ec(G_2)=\sum_{\mfrako'\in \mathcal{EO}(G_2)}\ec(G_2,\mfrako')$. Actually, there are $16$ Eulerian orientations of $G_2$, as shown in the following table. 

\begin{center}
\begin{tabular}{|c|c|c|c|}
\hline
Eulerian orientations $\mfrako'$ & $\In_w(G_2,\mfrako')$ & $\ec(G_2,\mfrako')$\\ \hline 
$\ve_0,\ve_1,\ve_2,\ve_3,\ve_4,\ve_5,\ve_6,\ve_7$ & 6 & 6  \\ \hline
$\ve_0,\ve_1,\ve_2,\ve_3,\ve_4^{-1},\ve_5,\ve_6^{-1},\ve_7^{-1}$ & 5 & 5  \\ \hline
$\ve_0,\ve_1,\ve_2,\ve_3,\ve_4^{-1},\ve_5^{-1},\ve_6,\ve_7^{-1}$ & 5 & 5  \\ \hline
$\ve_0,\ve_1^{-1},\ve_2,\ve_3^{-1},\ve_4^{-1},\ve_5,\ve_6^{-1},\ve_7$ & 6 & 6  \\ \hline
$\ve_0,\ve_1^{-1},\ve_2,\ve_3^{-1},\ve_4^{-1},\ve_5^{-1},\ve_6,\ve_7$ & 6 & 6  \\ \hline
$\ve_0,\ve_1^{-1},\ve_2^{-1},\ve_3,\ve_4,\ve_5,\ve_6,\ve_7$ & 6 & 6  \\ \hline
$\ve_0,\ve_1^{-1},\ve_2^{-1},\ve_3,\ve_4^{-1},\ve_5,\ve_6^{-1},\ve_7^{-1}$ & 5 & 5  \\ \hline
$\ve_0,\ve_1^{-1},\ve_2^{-1},\ve_3,\ve_4^{-1},\ve_5^{-1},\ve_6,\ve_7^{-1}$ & 5 & 5  \\ \hline
$\ve_0^{-1},\ve_1,\ve_2,\ve_3^{-1},\ve_4,\ve_5,\ve_6^{-1},\ve_7$ & 5 & 5  \\ \hline
$\ve_0^{-1},\ve_1,\ve_2,\ve_3^{-1},\ve_4,\ve_5^{-1},\ve_6,\ve_7$ & 5 & 5  \\ \hline
$\ve_0^{-1},\ve_1,\ve_2,\ve_3^{-1},\ve_4^{-1},\ve_5^{-1},\ve_6^{-1},\ve_7^{-1}$ & 6 & 6  \\ \hline
$\ve_0^{-1},\ve_1,\ve_2^{-1},\ve_3,\ve_4,\ve_5,\ve_6^{-1},\ve_7^{-1}$ & 6 & 6  \\ \hline
$\ve_0^{-1},\ve_1,\ve_2^{-1},\ve_3,\ve_4,\ve_5^{-1},\ve_6,\ve_7^{-1}$ & 6 & 6  \\ \hline
$\ve_0^{-1},\ve_1^{-1},\ve_2^{-1},\ve_3^{-1},\ve_4,\ve_5,\ve_6^{-1},\ve_7$ & 5 & 5  \\ \hline
$\ve_0^{-1},\ve_1^{-1},\ve_2^{-1},\ve_3^{-1},\ve_4,\ve_5^{-1},\ve_6,\ve_7$ & 5 & 5  \\ \hline
$\ve_0^{-1},\ve_1^{-1},\ve_2^{-1},\ve_3^{-1},\ve_4^{-1},\ve_5^{-1},\ve_6^{-1},\ve_7^{-1}$ & 6 & 6  \\ \hline
\end{tabular}
\end{center}

Note that in the table, the Eulerian orientations $\mfrako'$ are represented by the positive oriented edges with respect to $\mfrako'$. As a result, we get $\ec(G_2)=88$. 

Now the same result can also be computed using the trace formula in Theorem~\ref{T:main}. That is, 
$$\ec(G_2)= N_{2,F}(\mathbbm{1}_T,8)/8=\widetilde{N}_{2,F}(\mathbbm{1}_T,8)/8=88$$ where $F=E^c(T)=\{e_2,e_4,e_5,e_6,e_7\}$ and $\mathbbm{1}_T= 1\cdot e_2+1\cdot e_4+1\cdot e_5+1\cdot e_6+1\cdot e_7 \in M_2(F)$. 

\end{example}

\section{Reduction of computation}\label{S:reduction}
The computations in Theorem~\ref{T:main}, Theorem~\ref{T:counting0}, and Theorem~\ref{T:counting} can be reduced based on certain symmetries of the spectra distribution of the twisted adjacency matrices over $C_1(G,\mbbZ/t\mbbZ)$.

\subsection{Reduction based on spectral antisymmetry of twisted graph adjacency}
As discovered in \cite{LR2024spectral}, spectral antisymmetry of twisted graph adjacency is a phenomenon that the spectra of the twisted vertex and edge adjacency matrices of a graph $G$ have an antisymmetric distribution over the character group $\mcalX(G)$ of $G$ with two poles being the trivial character and a special character called the \emph{canonical character} (see Theorem~3.6(SA1) in \cite{LR2024spectral} for a precise statement). The canonical character is always a  $2$-torsion of $\mcalX(G)$, and in this section,  our discussion will be focused on the $2$-torsion subgroup of $\mcalX(G)$, which is non-canonically isomorphic to $M_2(E^c(T) )$ with $T$ being any spanning tree of $G$. The canonical character has a representative in $M_2(E^c(T) )$, which we call the canonical element in $M_2(E^c(T))$ defined as follows. 

\begin{definition}
For a spanning tree $T$ of $G$,  we say $\gamma_T=\sum_{e\in E^c(T)}\gamma_e\cdot e \in M_2(E^c(T) )$ is the \emph{canonical element in $M_2(E^c(T))$} if $\gamma_e=0$ whenever the length of the unique path on $T$ connecting the adjacent vertices of $e$ is odd, and $\gamma_e=1$ otherwise. 
\end{definition}

The above definition of the canonical element in $M_2(E^c(T))$ is essentially the same as Construction~3.3 in \cite{LR2024spectral}, and spectral antisymmetry in this special setting follows from Lemma~5.3 there. The result is stated in the following theorem.

\begin{theorem} \label{T:red_anti}
Let $G$ be a non-bipartite Eulerian graph of genus $g$ with $m$ edges, and $T$ a spanning tree of $G$. Let  $\gamma_T$ be the canonical element in $M_2(E^c(T) )$, and $\{S_1,S_2\}$ any partition of $M_2(E^c(T) )$ such that for any  $\gamma\in M_2(E^c(T) )$, $\gamma\neq \gamma+\gamma_T$ and if $\gamma\in S_1$, then $\gamma+\gamma_T\in S_2$. Then the number of Eulerian cycles on $G$ can be computed as follows:
\begin{align*}
\ec(G) &= \frac{1}{m\cdot 2^{g-1}}\sum_{\gamma\in S_i}(-1)^{\sigma(\gamma)}\trace\left(B_\gamma^m\right) \\
&= \frac{1}{m\cdot 2^{g-1}}\sum_{\gamma\in S_i}(-1)^{\sigma(\gamma)}\trace\left(A_\gamma^m\right)
\end{align*}
for $i=1,2$. 
\end{theorem}

Before proving Theorem~\ref{T:red_anti}, we will first prove a series of lemmas. For each $\gamma=\sum_{e\in E^c(T)}\gamma_e\cdot e\in M_2(E^c(T) )$, the \emph{support} of $\gamma$ is defined as $\supp(\gamma):=\{e\in E^c(T)|\gamma_e=1\}$. 

\begin{lemma}\label{L:bipartite}
The following are equivalent:
\begin{enumerate}[(i)]
\item $\gamma_T=0$;
\item $\supp(\gamma_T)= \emptyset$;
\item $G$ is bipartite. 
\end{enumerate}
\end{lemma}
\begin{proof}
The equivalence of (i) and (ii) directly follows from definition. 
Now note that the spanning tree $T$ is itself a bipartite graph, meaning that we can partition $V(G)$ into the disjoint union of $V_1$ and $V_2$ such that the length of unique path on $T$ connecting two vertices $v$ and $w$ is even if and only if both $v$ and $w$ belong to one of $V_1$ and $V_2$. Therefore, an edge $e$ belongs to the support of $\gamma_T$ if and only if the adjacent vertices of $e$ belong to one of $V_1$ and $V_2$. This implies the equivalence of (ii) and (iii). 
\end{proof}

\begin{lemma} \label{L:max_bipartite}
For a graph $G$ and a spanning tree $T$ of $G$, let $G'=G\setminus \supp(\gamma_T)$, i.e., $G'$ is derived by deleting the edges in $\supp(\gamma_T)$ from $G$. Then $G'$ is the maximal bipartite subgraph of $G$ containing $T$ as a subgraph. 
\end{lemma}
\begin{proof}
This is an immediate corollary of Lemma~\ref{L:bipartite}. 
\end{proof}

\begin{lemma} \label{L:parity}
For an Eulerian graph $G$ and a spanning tree $T$ of $G$, the cardinality of the support of $\gamma_T$ has the same parity as the number of edges of $G$. 
\end{lemma}

\begin{proof}
Let $G'=G\setminus \supp(\gamma_T)$. Then by Lemma~\ref{L:max_bipartite}, $G'$ is bipartite and $E(G)=\supp(\gamma_T)\sqcup E(G')$. Hence we just need to show $|E(G')|$ must be even. Suppose for contradiction that $|E(G')|$ is odd. Partition $V(G)$ into $V_1$ and $V_2$ such that each edge of $G'$ has one adjacent vertex in $V_1$ and the other in $V_2$. Since the sum of the degrees in $V_1$ is $|E(G')|$, the graph $G'$ has an odd number of vertices in $V_1$ with odd degree. Every edge in $\supp(\gamma_T)$ joins two vertices in the same part of this bipartition. Adding such an edge either changes no degree in $V_1$ or changes the degrees of two vertices in $V_1$, so the parity of the number of odd-degree vertices in $V_1$ is preserved. Thus $G$ has an odd-degree vertex and cannot be Eulerian, a contradiction. 
\end{proof}

\begin{lemma}[Spectral antisymmetry] \label{L:antisymmetry}
For a graph $G$, a spanning tree $T$ of $G$, the canonical element $\gamma_T$ in $M_2(E^c(T) )$, and any $\gamma\in M_2(E^c(T))$, we have 
\begin{enumerate}[(i)]
\item $\spec B_{\gamma+\gamma_T}=-\spec B_\gamma$ and $\spec A_{\gamma+\gamma_T} = -\spec A_\gamma$, and
\item for any length $l$, $\trace\left(B_{\gamma+\gamma_T}^l\right) = (-1)^l \trace\left(B_\gamma^l\right)$ and $\trace\left(A_{\gamma+\gamma_T}^l\right) = (-1)^l \trace\left(A_\gamma^l\right)$
\end{enumerate}
\end{lemma}

\begin{proof}
(i) is the specialization of Lemma~5.3 in \cite{LR2024spectral} to the 2-torsion subgroup of the character group of $G$.

(ii) follows from (i) and the fact that $\trace\left(B_\gamma^l\right)=\sum_{\lambda\in \spec B_\gamma}\lambda^l$ and $\trace\left(A_\gamma^l\right)=\sum_{\lambda\in \spec A_\gamma}\lambda^l$ (Lemma~\ref{L:trace}). 
\end{proof}

Now we can prove Theorem~\ref{T:red_anti}. 
\begin{proof}[Proof of Theorem~\ref{T:red_anti}]
By Lemma~\ref{L:bipartite}, since $G$ is non-bipartite, $\gamma_T\neq 0$, which means  $\gamma\neq \gamma+\gamma_T$ for all $\gamma\in M_2(E^c(T) )$. Consequently, for all $\gamma\in S_1$, $\gamma+\gamma_T\in S_2$ and vice versa. Moreover, by Lemma~\ref{L:parity} and Lemma~\ref{L:antisymmetry}, 
$$(-1)^{\sigma(\gamma+\gamma_T)}\trace\left(B_{\gamma+\gamma_T}^m\right)=(-1)^{\sigma(\gamma)+\sigma(\gamma_T)+m}\trace\left(B_\gamma^m\right) =(-1)^{\sigma(\gamma)}\trace\left(B_\gamma^m\right) $$
and 
$$(-1)^{\sigma(\gamma+\gamma_T)}\trace\left(A_{\gamma+\gamma_T}^m\right)=(-1)^{\sigma(\gamma)+\sigma(\gamma_T)+m}\trace\left(A_\gamma^m\right) =(-1)^{\sigma(\gamma)}\trace\left(A_\gamma^m\right).$$

This implies $\sum_{\gamma\in S_1}(-1)^{\sigma(\gamma)}\trace\left(B_\gamma^m\right) = \sum_{\gamma\in S_2}(-1)^{\sigma(\gamma)}\trace\left(B_\gamma^m\right)$
and $\sum_{\gamma\in S_1}(-1)^{\sigma(\gamma)}\trace\left(A_\gamma^m\right) = \sum_{\gamma\in S_2}(-1)^{\sigma(\gamma)}\trace\left(A_\gamma^m\right)$, and the theorem follows. 
\end{proof}

The following lemma provides a choice of $S_1$ and $S_2$ for effectively applying Theorem~\ref{T:red_anti} for computation. 
\begin{lemma}
Let $G$ be a non-bipartite Eulerian graph with a spanning tree $T$. Choose arbitrarily $e\in \supp(\gamma_T)$. Let $S_1=M_2(E^c(T)\setminus \{e\})$ and $S_2=M_2(E^c(T))\setminus S_1$. Then $S_2=\{\gamma+\gamma_T|\gamma\in S_1\}$ and $S_1=\{\gamma+\gamma_T|\gamma\in S_2\}$. 
\end{lemma}
\begin{proof}
First note that $\supp(\gamma_T)\neq\emptyset$ by Lemma~\ref{L:bipartite}. Without loss of generality, assume $E^c(T)=\{e_1,\cdots,e_g\}$, $\gamma_T = e_1+\cdots+e_k$ with some $1\leq k\leq g$, $S_1 = M_2(\{e_2,\cdots, e_g\})$ and $S_2 = M_2(\{e_1,\cdots, e_g\})\setminus M_2(\{e_2,\cdots, e_g\})$. Then for any $\gamma = c_2\cdot e_2+\cdots+c_g\cdot e_g\in S_1$ , $\gamma+\gamma_T=e_1+(c_2+1)\cdot e_2+\cdots +(c_k+1)\cdot e_k+c_{k+1}\cdot e_{k+1}+\cdots +c_g\cdot e_g\in S_2$. Similarly, for any $\gamma = e_1+c_2\cdot e_2+\cdots+c_g\cdot e_g\in S_2$, $\gamma+\gamma_T=(c_2+1)\cdot e_2+\cdots+(c_k+1)\cdot e_k+c_{k+1}\cdot e_{k+1}+\cdots +c_g\cdot e_g\in S_1$.
\end{proof}

\begin{example} \label{E:red_anti}
Reconsider Example~\ref{E:eulerian} for the Eulerian graph $G_2$  illustrated in Figure~\ref{F:example}(b). The maximal bipartite subgraph of $G_2$ containing $T$ has edges $e_0$, $e_1$, $e_2$, and $e_3$. Then the canonical element in $M_2(E^c(T) )=M_2(\{e_2,e_4,e_5,e_6,e_7\})$ is $\gamma_T=e_4 + e_5+ e_6 +e_7$, and $\supp(\gamma_T)=\{e_4,e_5,e_6,e_7\}$. Note that $|\supp(\gamma_T)|=4$ and $|E(G_2)|=8$, with the same parity  as Lemma~\ref{L:parity} shows.  Now we choose $S_1=M_2(\{e_2,e_4,e_5,e_6\})$ and $S_2=M_2(E^c(T) )\setminus S_1$. Then $\{S_1,S_2\}$ is a desirable partition of $M_2(E^c(T) )$ in the statement of Theorem~\ref{T:red_anti}, i.e., whenever $\gamma\in S_1$, $\gamma+\gamma_T\in S_2$. The following table shows all elements $\gamma$ in  $S_1$, their correspondence $\gamma+\gamma_T$ in $S_2$, and the values of 
 $(-1)^{\sigma(\gamma)}=  (-1)^{\sigma(\gamma+\gamma_T)}$, $\trace\left(B_\gamma^8\right)=(-1)^8\trace\left(B_{\gamma+\gamma_T}^8\right)=\trace\left(B_{\gamma+\gamma_T}^8\right)$ and $\trace\left(A_\gamma^8\right)=(-1)^8\trace\left(A_{\gamma+\gamma_T}^8\right)=\trace\left(A_{\gamma+\gamma_T}^8\right)$. 

\centering
\begin{tabular}{|c|c|c|c|c|}
\hline
$\gamma\in S_1$& $\gamma+\gamma_T\in S_2$  & $(-1)^{\sigma(\gamma)}$ & $\trace\left(B_\gamma^8\right)$ & $\trace\left(A_\gamma^8\right)$  \\ \hline 
0 & $e_4 + e_5+ e_6 +e_7$ & 1 & 6800 & 66048 \\ \hline 
$e_2$ & $e_2+e_4 + e_5+ e_6 +e_7$ & -1 & -112 & 12288 \\ \hline 
$e_4$ & $e_5+ e_6 +e_7$ & -1 & -624 & 20032 \\ \hline 
$e_2+e_4$ & $e_2+e_5+ e_6 +e_7$ & 1 & 144 & 6208 \\ \hline 
$e_5$ & $e_4+ e_6 +e_7$ & -1 & -112 & 12288 \\ \hline 
$e_2+e_5$ & $e_2+e_4+ e_6 +e_7$ & 1 & 400 & 512 \\ \hline 
$e_4+e_5$ & $e_6 +e_7$ & 1 & 144 & 6208 \\ \hline 
$e_2+e_4+e_5$ & $e_2+ e_6 +e_7$ & -1 & -624 & 64 \\ \hline 
$e_6$ & $e_4+e_5 +e_7$ & -1 & -112 & 12288 \\ \hline 
$e_2+e_6$ & $e_2+e_4+ e_5 +e_7$ & 1 & 400 & 512 \\ \hline 
$e_4+e_6$ & $e_5+ e_7$ & 1 & 144 & 6208 \\ \hline 
$e_2+e_4+e_6$ & $e_2+e_5 +e_7$ & -1 & -624 & 64 \\ \hline 
$e_5+e_6$ & $e_4+ e_7$ & 1 & 144 & 8704 \\ \hline 
$e_2+e_5+e_6$ & $e_2+e_4 +e_7$ & -1 & -112 & 12288 \\ \hline 
$e_4+e_5+e_6$ & $e_7$ & -1 & -624 & 20032 \\ \hline 
$e_2+e_4+e_5+e_6$ & $e_2+ e_7$ & 1 & 144 & 6208 \\ \hline 
\end{tabular}
\end{example}

\subsection{Reduction based on graph automorphisms}
\begin{definition}
An \emph{automorphism} $\tau$ of $G$ is composed of a permutation of $V(G)$ and a permutation of $E(G)$ (both permutations denoted by $\tau$ by an abuse of notation) such that the adjacency relations are preserved, i.e., for all $v\in V(G)$ and $e\in E(G)$, $v$ is adjacent to $e$ if and only if $\tau(v)$ is adjacent to $\tau(e)$. All automorphisms of $G$ form a group which is called the \emph{automorphism group} of $G$, denoted by $\Aut(G)$. 

If an orientation $\mfrako$ of $G$ is fixed, let $\Aut(G,\mfrako)$ denote the subgroup of $\Aut(G)$ consisting of automorphisms which send every positively oriented edge with respect to $\mfrako$ to a positively oriented edge with respect to $\mfrako$. Equivalently, $\Aut(G,\mfrako)$ is the automorphism group of the directed graph $G_\mfrako$.
\end{definition}

\begin{remark} \label{R:auto}
The actions of an automorphism $\tau$ of $G$ on $V(G)$ and $E(G)$ can be naturally extended to actions on $\vE(G)$ and $C_1(G,R)$ as follows: 
\begin{enumerate}[(i)]
\item For each pair of oppositely oriented edges $\ve,\ve^{-1}\in \vE(G)$ with the same underlying edge $e\in E(G)$, $\tau(\ve)$ and $\tau(\ve^{-1})$  are the oppositely oriented edges with the same underlying edge $\tau(e)$ such that their initial vertices are $\tau(\ve(0))$ and $\tau(\ve^{-1}(0))$ respectively. Clearly such an action on $\vE(G)$ is a permutation. However, it does not preserve orientation in general. That is to say a positively oriented edge $\ve\in \vE_\mfrako(G)$ with respect to an orientation $\mfrako$ can possibly be mapped to a negatively oriented edge $\tau(\ve)\in  \vE(G)\setminus \vE_{\mfrako}(G)$  and vice versa. In this case, we say the action of \emph{$\tau$ changes the orientation of $\ve$} with respect to $\mfrako$; otherwise, we say \emph{$\tau$ preserves the orientation of $\ve$} with respect to $\mfrako$.
\item For each $\gamma=\sum_{\ve\in \vE_\mfrako(G)}c_{\ve}\cdot \ve \in C_1(G,R)$, the action of $\tau$ on $\gamma$ is defined to be $\tau(\gamma)=\sum_{\ve\in \vE_\mfrako(G)}c_{\ve}\cdot \tau(\ve)$. Alternatively, we can write $\tau(\gamma) = \sum_{\ve'\in \vE_\mfrako(G)}b_{\ve'}\cdot \ve'$ where for each $\ve'\in \vE_\mfrako(G)$, $b_{\ve'}=c_{\ve}$ if $\ve'=\tau^{-1}(\ve)$ (i.e., $\tau$ preserves the direction of $\ve$) and $b_{\ve'}=-c_{\ve}$ if $\ve'=(\tau^{-1}(\ve))^{-1}$ (i.e., $\tau$ changes the direction of $\ve$). One can also easily verify that such an action of $\tau$ on $C_1(G,R)$ is also an automorphism of $C_1(G,R)$. 
\end{enumerate}
\end{remark}

The adjacency matrix is a graph isomorphic invariant (more precisely, for labelled graphs). The following lemma shows that the conjugacy of twisted vertex and edge adjacency matrices is invariant under automorphisms. 
\begin{lemma} \label{L:auto}
Let $\tau$ be an automorphism of $G$. For each $\gamma\in C_1(G,\mbbZ/t\mbbZ)$, $A_\gamma$ is similar to $A_{\tau(\gamma)}$ and $B_\gamma$ is similar to $B_{\tau(\gamma)}$.
\end{lemma}

\begin{proof}
Let $\mfrako$ be an orientation of $G$. 
Index the vertices and oriented edges of $G$ as $V(G)=\{v_1,\cdots,v_n\}$, $\vE_\mfrako(G)=\{\ve_1,\cdots,\ve_m\}$ and $\vE(G)=\{\ve_1,\cdots,\ve_m,\ve_{m+1}=\ve_1^{-1},\cdots,\ve_{2m}=\ve_m^{-1}\}$.  Let $\gamma=c_1\cdot \ve_1+\cdots+c_m\cdot \ve_m\in C_1(G,\mbbZ/t\mbbZ)$. Then $\tau(\gamma)=c_1\cdot \tau(\ve_1)+\cdots+c_m\cdot \tau(\ve_m)$. This also means $\langle \gamma,\ve\rangle = \langle \tau(\gamma),\tau(\ve)\rangle$ for each $\ve\in \vE(G)$. 

By definition, the twist vertex adjacency matrices are $A_\gamma=(a_{ij})_{n\times n}$ with $a_{ij}=\sum_{\ve\in \vE(G),\ve(0)=v_i,\ve(1)=v_j} \epsilon_t^{\langle\gamma,\ve\rangle}$, and $A_{\tau(\gamma)}=(a'_{ij})_{n\times n}$ with $a'_{ij}=\sum_{\ve\in \vE(G),\ve(0)=v_i,\ve(1)=v_j} \epsilon_t^{\langle\tau(\gamma),\ve\rangle}$. 

Let $v_{i'}=\tau(v_i)$ and $v_{j'}=\tau(v_j)$. Then since the action of $\tau$  on $\vE(G)$ is a permutation respecting the incidence relations and the identity $\langle \gamma,\ve\rangle = \langle \tau(\gamma),\tau(\ve)\rangle$ holds for all $\ve\in \vE(G)$, we have
\begin{align*}
a'_{i'j'}&=\sum_{\ve\in \vE(G),\ve(0)=\tau(v_i),\ve(1)=\tau(v_j)} \epsilon_t^{\langle\tau(\gamma),\ve\rangle} \\
&=\sum_{\ve\in \vE(G),\tau(\ve)(0)=\tau(v_i),\tau(\ve)(1)=\tau(v_j)} \epsilon_t^{\langle\tau(\gamma),\tau(\ve)\rangle} \\
&=\sum_{\ve\in \vE(G),\ve(0)=v_i,\ve(1)=v_j} \epsilon_t^{\langle\gamma,\ve\rangle} = a_{ij}. 
\end{align*}
This means that the matrix $A_{\tau(\gamma)}$ can be derived from $A_\gamma$ by a permutation of rows and columns with respect to the permutation of $\tau$ on $V(G)$. Hence $A_\gamma$ is similar to $A_{\tau(\gamma)}$. 

For the twisted edge adjacency matrices, $B_\gamma=(b_{ij})_{2m\times 2m}$ (respectively $B_{\tau(\gamma)}=(b'_{ij})_{2m\times 2m}$) such that $b_{ij} =\epsilon_t^{\langle\gamma,\ve_i\rangle}$ (respectively $b'_{ij} =\epsilon_t^{\langle\tau(\gamma),\ve_i\rangle}$) if $\ve_i$ feeds into $\ve_j$, and $b_{ij}=0$ ($b'_{ij}=0$ ) otherwise.

Let $\ve_{i'}=\tau(\ve_i)$ and $\ve_{j'}=\tau(\ve_j)$. $\tau$ being an isomorphism guarantees that $\ve_{i'}$ feeds into $\ve_{j'}$ if and only if $\ve_i$ feeds into $\ve_j$. Then $b'_{i'j'}=\epsilon_t^{\langle\tau(\gamma),\ve_{i'}\rangle}=\epsilon_t^{\langle\tau(\gamma),\tau(\ve_{i})\rangle} = \epsilon_t^{\langle\gamma,\ve_i\rangle}=b_{ij}$ if $\ve_i$ feeds into $\ve_j$, and $b'_{i'j'}=b_{ij}=0$ otherwise. Consequently, the matrix $B_{\tau(\gamma)}$ can be derived from $B_\gamma$ by a permutation of rows and columns with respect to the permutation of $\tau$ on $\vE(G)$. Hence $B_\gamma$ is similar to $B_{\tau(\gamma)}$. 
\end{proof}

\begin{example}
Consider the Eulerian graph $G_2$  in Figure~\ref{F:example}(b). Let $\tau$ be an automorphism of $G_2$ defined by  $v_0\leftrightarrow v_2$, $v_1\leftrightarrow v_3$, $e_1\leftrightarrow e_2$, $e_5\leftrightarrow e_6$, $e_0\leftrightarrow e_7$, $e_3\leftrightarrow e_4$, where the sign ``$\leftrightarrow$'' means ``transposition''. Then $\tau$ acts on the oriented edges as follows: $\tau(\ve_0)=\ve_7^{-1}$, $\tau(\ve_1)=\ve_2$, $\tau(\ve_2)=\ve_1$, $\tau(\ve_3)=\ve_4^{-1}$, $\tau(\ve_4)=\ve_3^{-1}$, $\tau(\ve_5)=\ve_6^{-1}$, $\tau(\ve_6)=\ve_5^{-1}$, $\tau(\ve_7)=\ve_0^{-1}$. Now consider $\gamma = 1\cdot \ve_1+ 2\cdot \ve_2 + 3\cdot \ve_3 +2\cdot \ve_5 +1\cdot \ve_6+3\cdot \ve_7 \in C_1(G_2,\mbbZ/4\mbbZ)$. Then 
\begin{align*}
\tau(\gamma)&= \tau(1\cdot \ve_1+ 2\cdot \ve_2 + 3\cdot \ve_3 +2\cdot \ve_5 +1\cdot \ve_6+3\cdot \ve_7) \\
&=1\cdot \tau(\ve_1)+ 2\cdot \tau(\ve_2) + 3\cdot \tau(\ve_3) +2\cdot \tau(\ve_5) +1\cdot \tau(\ve_6)+3\cdot \tau(\ve_7) \\
&=1\cdot \ve_2 + 2\cdot \ve_1-3\cdot \ve_4 -2\cdot \ve_6-1\cdot \ve_5 -3\cdot \ve_0 \\
& = 1\cdot \ve_0+2\cdot \ve_1 +1\cdot \ve_2 +1\cdot \ve_4 +3\cdot \ve_5+2\cdot \ve_6.
\end{align*}
With respect to the indexing of vertices $v_0,\cdots,v_3$ and oriented edges $\ve_0,\cdots,\ve_7,\ve_8=\ve_0^{-1},\cdots,\ve_{15}=\ve_7^{-1}$, we may compute the twisted vertex adjacency matrices $A_\gamma$ and $A_{\tau(\gamma)}$,  and the twisted edge adjacency matrices $B_\gamma$ and $B_{\tau(\gamma)}$ as follows: 
$A_\gamma=\begin{bmatrix}
    0    &  1  & - 1-i  &   -i \\
    1    &  0    &  1 & - 1-i \\
 - 1+i &   1   &   0   &  -i\\
    i &  -1+i &     i &      0
\end{bmatrix}$
is similar to  \\
$A_{\tau(\gamma)}=\begin{bmatrix}
    0    &  -i & -1+i  &   1 \\
    i    &  0    &  i & -1+i \\
 - 1-i &   -i   &   0   &  1\\
    1 &  -1-i &     1 &      0
\end{bmatrix}$, and 
$B_\gamma= \\
\begin{bmatrix}
 0 & 0&  1&  1&  0&  0&  0&  0&  0&  1&  0&  0&  0&  0&  0&  0 \\
 0 & 0 & i & i & 0 & 0 & 0 & 0&  i&  0&  0&  0&  0&  0&  0&  0\\
0 &-1&  0&  0&  0&  0&  0& -1&  0&  0&  0&  0& -1&  0&  0&  0\\
 0 & 0 & 0 & 0 & 0& -i& -i&  0&  0&  0&  0&  0&  0&  0&  0& -i \\
 0 & 1 & 0 & 0 & 0 & 0 & 0 & 1 & 0 & 0&  1&  0&  0&  0&  0&  0 \\
-1 & 0 & 0 & 0& -1&  0&  0&  0&  0&  0&  0&  0&  0&  0& -1 & 0 \\
 i  &0 & 0 & 0 & i & 0&  0&  0&  0&  0&  0&  0&  0&  i&  0&  0 \\
 0 & 0&  0&  0&  0& -i& -i&  0 & 0 & 0 & 0 &-i&  0&  0&  0&  0 \\
 0 & 0&  0&  0&  1&  0&  0&  0&  0&  0&  0&  0&  0&  1&  1&  0 \\
 0 & 0 & 0 & 0 & 0 & 0 & 0 &-i & 0&  0& -i&  0& -i & 0&  0 & 0 \\
 0  &0 & 0& -1&  0&  0&  0&  0& -1& -1&  0 & 0 & 0 & 0 & 0 & 0 \\
 0 & 0&  i&  0 & 0&  0&  0&  0&  i&  i&  0&  0&  0&  0&  0&  0 \\
 1 & 0 & 0 & 0 & 0 & 0 & 0 & 0 & 0 & 0 & 0 & 0 & 0&  1&  1&  0 \\
 0 & 0 & 0 & 0 & 0 & 0 &-1 & 0 & 0 & 0 & 0& -1&  0&  0&  0& -1 \\
 0 & 0 & 0 & 0 & 0 &-i & 0 & 0 & 0 & 0 & 0 &-i & 0 & 0 & 0& -i \\
 0 & i&  0 & 0 & 0 & 0 & 0 & 0 & 0 & 0 & i & 0 & i & 0 & 0 & 0
\end{bmatrix}$ is similar to 
$B_{\tau(\gamma)} = \\
\begin{bmatrix}
 0 & 0&  i&  i&  0&  0&  0&  0&  0&  i&  0&  0&  0&  0&  0&  0 \\
 0 & 0 & -1 & -1 & 0 & 0 & 0 & 0&  -1&  0&  0&  0&  0&  0&  0&  0\\
0 &i&  0&  0&  0&  0&  0& i&  0&  0&  0&  0& i&  0&  0&  0\\
 0 & 0 & 0 & 0 & 0& 1& 1&  0&  0&  0&  0&  0&  0&  0&  0& 1 \\
 0 & i & 0 & 0 & 0 & 0 & 0 & i& 0 & 0& i&  0&  0&  0&  0&  0 \\
-i & 0 & 0 & 0& -i &  0&  0&  0&  0&  0&  0&  0&  0&  0& -i & 0 \\
 -1  &0 & 0 & 0 & -1 & 0&  0&  0&  0&  0&  0&  0&  0&  -1&  0&  0 \\
 0 & 0&  0&  0&  0& 1& 1&  0 & 0 & 0 & 0 &1&  0&  0&  0&  0 \\
 0 & 0&  0&  0&  -i&  0&  0&  0&  0&  0&  0&  0&  0&  -i&  -i&  0 \\
 0 & 0 & 0 & 0 & 0 & 0 & 0 &-1 & 0&  0& -1&  0& -1 & 0&  0 & 0 \\
 0  &0 & 0& -i&  0&  0&  0&  0& -i& -i&  0 & 0 & 0 & 0 & 0 & 0 \\
 0 & 0&  1&  0 & 0&  0&  0&  0&  1&  1&  0&  0&  0&  0&  0&  0 \\
 -i  & 0 & 0 & 0 & 0 & 0 & 0 & 0 & 0 & 0 & 0 & 0 & 0&  -i&  -i &  0 \\
 0 & 0 & 0 & 0 & 0 & 0 &i & 0 & 0 & 0 & 0& i&  0&  0&  0& i \\
 0 & 0 & 0 & 0 & 0 &-1 & 0 & 0 & 0 & 0 & 0 &-1 & 0 & 0 & 0& -1\\
 0 & 1&  0 & 0 & 0 & 0 & 0 & 0 & 0 & 0 & 1& 0 & 1& 0 & 0 & 0
\end{bmatrix}$.

\qed
\end{example}

Now let $H$ be any subgroup of $\Aut(G)$. By Remark~\ref{R:auto}(ii), every element of $H$ induces an automorphism of $C_1(G,\mbbZ/t\mbbZ)$. For a subset $S$ of $C_1(G,\mbbZ/t\mbbZ)$, denote by $S/H$ the partition of $S$ obtained by intersecting full $H$-orbits in $C_1(G,\mbbZ/t\mbbZ)$ with $S$. Thus
$$
S/H=\{O\cap S: O\in C_1(G,\mbbZ/t\mbbZ)/H,\ O\cap S\neq \emptyset\}.
$$
When $\gamma\in S$, we write $[\gamma]_S=H\gamma\cap S$ for the corresponding part of this partition. If $S$ is clear from the context, we simply write $[\gamma]$. This convention is important because the cotree coordinate set $M_t(\vE_\mfrako^c(T))$ need not be preserved by every automorphism of $G$.

\begin{theorem}\label{T:red_counting}
For a subgroup $H$ of $\Aut(G)$, a subset $\vF$ of $\vE_\mfrako(G)$ and $\alpha\in M_t(\vF)$,  we have
$$N_{t,\vF}(\alpha,l) = \frac{1}{t^{|\vF|}}\sum_{[\gamma]\in M_t(\vF)/H}\left(\sum_{\gamma'\in [\gamma]}\epsilon_t^{-\langle\gamma', \alpha\rangle}\right)\trace\left(B_\gamma^l\right)$$
and
$$ \widetilde{N}_{t,\vF}(\alpha,l)= \frac{1}{t^{|\vF|}}\sum_{[\gamma]\in M_t(\vF)/H}\left(\sum_{\gamma'\in [\gamma]}\epsilon_t^{-\langle\gamma', \alpha\rangle}\right)\trace\left(A_{\gamma}^l\right)$$
for any length $l$. If $G$ has genus $g$, then for any spanning tree $T$ of $G$, we have 
$$N_t(\alpha,l) = \frac{1}{t^g}\sum_{[\gamma]\in M_t(\vE_\mfrako^c(T))/H}\left(\sum_{\gamma'\in [\gamma]}\epsilon_t^{-\langle\gamma', \alpha\rangle}\right)\trace\left(B_\gamma^l\right)$$
and
$$ \widetilde{N}_t(\alpha,l)= \frac{1}{t^g}\sum_{[\gamma]\in M_t(\vE_\mfrako^c(T))/H}\left(\sum_{\gamma'\in [\gamma]}\epsilon_t^{-\langle\gamma', \alpha\rangle}\right)\trace\left(A_{\gamma}^l\right)$$
for any length $l$.
\end{theorem}

\begin{proof}
By Lemma~\ref{L:auto}, for any orbit $[\gamma]$ and any $\gamma'\in [\gamma]$, we have $\trace\left(B_\gamma^l\right)=\trace\left(B_{\gamma'}^l\right)$ and  $\trace\left(A_{\gamma}^l\right)=\trace\left(A_{\gamma'}^l\right)$. Hence the theorem follows from a reformulation of Theorem~\ref{T:counting0} and Theorem~\ref{T:counting}. 
\end{proof}

\begin{theorem}\label{T:red_counting_eulerian}
Let $G$ be an Eulerian graph of genus $g$ with $m$ edges, $T$ a spanning tree of $G$, and $H$ a subgroup of $\Aut(G)$. We have
\begin{align*}
\ec(G) &=  \frac{1}{m\cdot 2^g}\sum_{[\gamma]\in M_2(E^c(T))/H}|[\gamma]|\cdot(-1)^{\sigma(\gamma)}\trace\left(B_\gamma^m\right) \\
&= \frac{1}{m\cdot 2^g}\sum_{[\gamma]\in M_2(E^c(T))/H}|[\gamma]|\cdot(-1)^{\sigma(\gamma)}\trace\left(A_{\gamma}^m\right)
 \end{align*}
where the orbits are understood as restricted orbit intersections in $M_2(E^c(T))$. Moreover, if $\mfrako$ is an Eulerian orientation of $G$ and $H$ is a subgroup of $\Aut(G,\mfrako)$, then
 \begin{align*}
\ec(G,\mfrako) &=  \frac{1}{m\cdot t^g}\sum_{[\gamma]\in M_t(\vE_\mfrako^c(T))/H}|[\gamma]| \cdot\epsilon_t^{-\sigma(\gamma)}\trace\left(B_\gamma^m\right) \\
&= \frac{1}{m\cdot t^g}\sum_{[\gamma]\in M_t(\vE_\mfrako^c(T))/H}|[\gamma]|\cdot\epsilon_t^{-\sigma(\gamma)}\trace\left(A_{\gamma}^m\right)
 \end{align*}
 for all $t\geq 3$. 
\end{theorem}
\begin{proof}
For the undirected formula, $t=2$ and the all-one mod-$2$ chain $\mathbbm{1}$ is invariant under every automorphism of the underlying graph. Hence $(-1)^{\sigma(\gamma)}=(-1)^{\sigma(\tau(\gamma))}$ for every $\tau\in H$, and the result follows by grouping the terms in Theorem~\ref{T:main} according to the restricted orbit intersections defined above.

For the directed formula with $t\geq 3$, the same simplification is valid only when the automorphisms preserve the chosen orientation $\mfrako$. If $\tau\in\Aut(G,\mfrako)$, then $\tau$ preserves the circulation $\mathbbm{1}_\mfrako$, and therefore
$$
\epsilon_t^{-\sigma(\gamma)}
=\epsilon_t^{-\langle\gamma,\mathbbm{1}_\mfrako\rangle}
=\epsilon_t^{-\langle\tau(\gamma),\mathbbm{1}_\mfrako\rangle}
=\epsilon_t^{-\sigma(\tau(\gamma))}.
$$
The directed identities then follow from Theorem~\ref{T:directed_eulerian} by the same orbit grouping.
\end{proof}

To combine the reductions based on spectral antisymmetry and graph automorphisms for a non-bipartite Eulerian graph $G$, we propose two approaches here: the first is to simply restrict the action of group automorphisms to only one of $S_1$ or $S_2$ in the statement of Theorem~\ref{T:red_anti}, and the second is to combine the effects of the permutations of $M_2(E^c(T) )$ induced by spectral antisymmetry and graph automorphisms. 
In particular, for the latter, consider a subgroup $H$ of $\Aut(G)$ and a subset $S$ of $M_2(E^c(T) )$. Let $\langle H,\gamma_T\rangle$ be the permutation group of $C_1(G,\mbbZ/2\mbbZ)$ generated by the automorphisms in $H$ and the translation permutation $\gamma\mapsto \gamma+\gamma_T$. We write $S/\langle H,\gamma_T\rangle$ for the partition of $S$ by intersections with the full $\langle H,\gamma_T\rangle$-orbits.

\begin{theorem}\label{T:red_counting_eulerian_combined}
Let $G$ be a non-bipartite Eulerian graph of genus $g$ with $m$ edges, $T$ a spanning tree of $G$,  $\gamma_T$ be the canonical element in $M_2(E^c(T) )$ and $H$ a subgroup of $\Aut(G)$. Let $S_1\sqcup S_2$ be any partition of $M_2(E^c(T) )$ such that for any  $\gamma\in M_2(E^c(T) )$, $\gamma\neq \gamma+\gamma_T$ and if $\gamma\in S_1$, then $\gamma+\gamma_T\in S_2$. Then we have
\begin{align*}
\ec(G) &= \frac{1}{m\cdot 2^{g-1}}\sum_{[\gamma]\in S_i/H}|[\gamma]|\cdot(-1)^{\sigma(\gamma)}\trace\left(B_\gamma^m\right) \\
&= \frac{1}{m\cdot 2^{g-1}}\sum_{[\gamma]\in S_i/H}|[\gamma]|\cdot(-1)^{\sigma(\gamma)}\trace\left(A_\gamma^m\right) \\
&=  \frac{1}{m\cdot 2^g}\sum_{[\gamma]\in M_2(E^c(T))/\langle H,\gamma_T\rangle}|[\gamma]|\cdot(-1)^{\sigma(\gamma)}\trace\left(B_\gamma^m\right) \\
&= \frac{1}{m\cdot 2^g}\sum_{[\gamma]\in M_2(E^c(T))/\langle H,\gamma_T\rangle}|[\gamma]|\cdot(-1)^{\sigma(\gamma)}\trace\left(A_{\gamma}^m\right) \\
&=  \frac{1}{m\cdot 2^{g-1}}\sum_{[\gamma]\in S_i/\langle H,\gamma_T\rangle}|[\gamma]|\cdot(-1)^{\sigma(\gamma)}\trace\left(B_\gamma^m\right) \\
&= \frac{1}{m\cdot 2^{g-1}}\sum_{[\gamma]\in S_i/\langle H,\gamma_T\rangle}|[\gamma]|\cdot(-1)^{\sigma(\gamma)}\trace\left(A_{\gamma}^m\right)
\end{align*}
for $i=1,2$. 
\end{theorem}
\begin{proof}
Let
$$
q_B(\gamma)=(-1)^{\sigma(\gamma)}\trace(B_\gamma^m),\qquad
q_A(\gamma)=(-1)^{\sigma(\gamma)}\trace(A_\gamma^m).
$$
By Theorem~\ref{T:red_counting_eulerian}, both $q_B$ and $q_A$ are constant on the restricted $H$-orbits in $M_2(E^c(T))$. By Theorem~\ref{T:red_anti}, they are also unchanged by the translation $\gamma\mapsto \gamma+\gamma_T$. Therefore they are constant on the restricted $\langle H,\gamma_T\rangle$-orbits. The formulas follow by grouping the sums in Theorem~\ref{T:main} and Theorem~\ref{T:red_anti} according to these orbit intersections. 

\end{proof}

\begin{example}
Let us consider the Eulerian graph $G_2$  in Figure~\ref{F:example}(b) again. As shown in Example~\ref{E:red_anti}, $M_2(E^c(T))=M_2(\{e_2,e_4,e_5,e_6,e_7\})$, $\gamma_T=e_4 + e_5+ e_6 +e_7$, and we choose $S_1=M_2(\{e_2,e_4,e_5,e_6\})$ and $S_2=M_2(E^c(T) )\setminus S_1$. Let $H=\langle\tau_1,\tau_2,\tau_3,\tau_4,\tau_5\rangle$ be a group of automorphisms of $G$ where 
\begin{align*}
\tau_1: & e_1 \leftrightarrow e_2 \\
\tau_2: & e_5\leftrightarrow e_6 \\
\tau_3: & v_1\leftrightarrow v_3, e_0\leftrightarrow e_3, e_4\leftrightarrow e_7 \\
\tau_4: & v_0\leftrightarrow v_2, e_0 \leftrightarrow e_4, e_3\leftrightarrow e_7 \\
\tau_5: & v_0\leftrightarrow v_1, v_2\leftrightarrow v_3, e_1\leftrightarrow e_5,e_2\leftrightarrow e_6, e_3\leftrightarrow e_4. 
\end{align*}
The following table shows the partitions $M_2(E^c(T))/H$, $S_1/H$ and $S_2/H$  of $M_2(E^c(T))$, $S_1$ and $S_2$ induced by $H$ respectively, corresponding to Theorem~\ref{T:red_counting_eulerian} and the first two identities of Theorem~\ref{T:red_counting_eulerian_combined}. Note that for each $[\gamma]\in M_2(E^c(T))/H$ and $i=1,2$, we have $[\gamma]\cap S_i\in S_i/H$ if $[\gamma]\cap S_i\neq \emptyset$. In particular, $|M_2(E^c(T))/H|=15$, $|S_1/H|=|S_2/H|=10$.  

\begin{center}
\begin{tabular}{|c|c|}
\hline
\multicolumn{2}{|c|}{$[\gamma]\in M_2(E^c(T))/H$} \\ \hline
$[\gamma]\cap S_1$& $[\gamma]\cap S_2$   \\ \hline 
0 &  \\ \hline 
$e_2$, $e_5$, $e_6$ & \\ \hline
$e_4$   &   $e_7$   \\ \hline
$e_2+e_4$,  $e_4+e_5$, $e_4+e_6$ & $e_2+e_7$, $e_5+e_7$,  $e_6+e_7$  \\ \hline
$e_2+e_5$, $e_2+e_6$  &    \\ \hline
$e_5+e_6$ &    \\ \hline
& $e_4+e_7$    \\ \hline
$e_2+e_4+e_5$, $e_2+e_4+e_6$  & $e_2+e_5+e_7$, $e_2+e_6+e_7$ \\ \hline
$e_2+e_5+e_6$ &\\ \hline
$e_4+e_5+e_6$ & $e_5+e_6+e_7$\\ \hline
& $e_2+e_4+e_7$, $e_4+e_5+e_7$,  $e_4+e_6+e_7$\\ \hline
$e_2+e_4+e_5+e_6$ & $e_2+e_5+e_6+e_7$\\ \hline
&$e_2+e_4+e_5+e_7$, $e_2+e_4+e_6+e_7$\\ \hline
&$e_4+e_5+e_6+e_7$ \\ \hline
& $e_2+e_4+e_5+e_6+e_7$\\ \hline
\end{tabular}
\end{center}

Furthermore, corresponding to the last four identities of Theorem~\ref{T:red_counting_eulerian_combined}, the next table shows the partitions $M_2(E^c(T))/\langle H,\gamma_T\rangle$, $S_1/\langle H,\gamma_T\rangle$ and $S_2/\langle H,\gamma_T\rangle$  of $M_2(E^c(T))$, $S_1$ and $S_2$ induced by $\langle H,\gamma_T\rangle$ respectively. Now $|M_2(E^c(T))/\langle H,\gamma_T\rangle|=|S_1/\langle H,\gamma_T\rangle|=|S_2/\langle H,\gamma_T\rangle|=7$, while there are exactly 7 different types of values of tuples $\left((-1)^{\sigma(\gamma)}\trace\left(B_\gamma^8\right),(-1)^{\sigma(\gamma)}\trace\left(A_\gamma^8\right)\right)$ as shown in the table in  Example~\ref{E:red_anti}. 
\begin{center}
\begin{tabular}{|c|c|}
\hline
\multicolumn{2}{|c|}{$[\gamma]\in M_2(E^c(T))/\langle H,\gamma_T\rangle$} \\ \hline
$[\gamma]\cap S_1$& $[\gamma]\cap S_2$    \\ \hline 
0 & $e_4+e_5+e_6+e_7$  \\ \hline 
$e_2$, $e_5$,  & $e_2+e_4+e_5+e_6+e_7$, $e_4+ e_6 +e_7$, \\  
$e_6$,  $e_2+e_5+e_6$ & $e_4+e_5 +e_7$, $e_2+ e_4 +e_7$ \\  \hline
$e_4$, $e_4+e_5+e_6$   &  $e_5+e_6+e_7$,  $e_7$    \\ \hline
$e_2+e_4$,  $e_4+e_5$,  &  $e_2+e_5+ e_6 +e_7$, $e_6+e_7$,  \\ 
 $e_4+e_6$, $e_2+e_4+e_5+e_6$  & $e_5+e_7$,   $e_2+e_7$  \\ \hline
$e_2+e_5$, $e_2+e_6$  & $e_2+e_4+ e_6 +e_7$, $e_2+e_4+ e_5 +e_7$   \\ \hline
$e_5+e_6$ & $e_4+ e_7$   \\ \hline
$e_2+e_4+e_5$, $e_2+e_4+e_6$  & $e_2+e_5+e_7$, $e_2+e_6+e_7$ \\ \hline
\end{tabular}
\end{center}

\end{example}

\section{Concluding remarks}

The formula in Theorem~\ref{T:main} should be read both as an exact enumeration formula and as a Fourier filter over the mod-$2$ homology of the graph.  This viewpoint explains why the genus is the natural parameter: the number of twists is the size of $H_1(G,\mbbZ/2\mbbZ)$.  Since exact counting is \#P-complete in general, the value of the formula is not that it turns the whole problem into an unrestricted polynomial-time computation.  Rather, it identifies a concrete source of exponential complexity and confines it to the homological parameter $g$, while the remaining trace computations are polynomial for fixed $g$ and can be reduced further by the symmetries in Section~\ref{S:reduction}.

The trace form also separates two computational questions.  The first is how to reduce the number of twists; this is addressed here by spectral antisymmetry, automorphism orbits, and their combined action.  The second is how to evaluate each trace efficiently.  For this purpose, the vertex formulation is especially useful because $A_\gamma$ is an $n$ by $n$ Hermitian matrix, whereas the edge matrix $B_\gamma$ has size $2m$ and is generally nonnormal.  Characteristic-polynomial trace recovery, sparse linear algebra, and determinant identities of Bass--Ihara type are natural ways to implement the individual trace evaluations without changing the homological character-sum structure.

Several further directions remain open.  One direction is to sharpen the reduction step by finding canonical representatives for the combined antisymmetry-automorphism action, especially for graph families with large symmetry groups.  Another is to compare the trace formula with existing exact methods on restricted classes, such as generalized series-parallel graphs, where polynomial-time algorithms are already known.  Randomized trace estimation or sparse Fourier methods may also be useful for approximation on large graphs, but the alternating character sum requires careful control of cancellation errors.  Applications to assembly graphs~\cite{PTW2001eulerian,CPT2011debruijn}, route-redundancy questions~\cite{EJ1973matching}, and other low-genus networks are promising because the exact formula is practical when $g$ is small; however, such applications require domain-specific modeling of how graph simplification affects the Eulerian-cycle count.

\bibliographystyle{alpha}
\bibliography{citation}
\end{document}